%% file: neutral.tex
\documentclass[11pt]{article}
\usepackage{amsthm, times}

\input{math_commands.tex}

\input{preamble_2.tex}

\graphicspath{{./figures/}}

\usepackage{booktabs} 
\usepackage{breakcites}
\usepackage{authblk}
\usepackage[margin=1in]{geometry}

\title{Average-case Acceleration for Bilinear Games and Normal Matrices}

\author[a]{Carles Domingo-Enrich}
\author[b]{Fabian Pedregosa}
\author[c]{Damien Scieur}
\affil[a]{Courant Institute of Mathematical Sciences, New York University}
\affil[b]{Google Research}
\affil[c]{Samsung SAIL Montreal}

\makeatletter
\def\imod#1{\allowbreak\mkern10mu({\operator@font mod}\,\,#1)}
\makeatother

\makeatletter
\def\inmod#1{\allowbreak\mkern5mu({\operator@font mod}\,\,#1)}
\makeatother
\usepackage{savesym}
\usepackage{xcolor}
\usepackage{amsthm}
\usepackage{amsmath}
\usepackage{amssymb}

\usepackage{hyperref}

\hypersetup{
     colorlinks = true,
     linkcolor = brown,
     anchorcolor = blue,
     citecolor = teal,
     filecolor = blue,
     urlcolor = blue
     }

\begin{document}

\maketitle

\begin{abstract}
Advances in generative modeling and adversarial learning have given rise to renewed interest in smooth games. However, the absence of symmetry in the matrix of second derivatives poses challenges that are not present in the classical minimization framework. While a rich theory of average-case analysis has been developed for minimization problems, little is known in the context of smooth games. In this work we take a first step towards closing this gap by developing average-case optimal first-order methods for a subset of smooth games. 
We make the following three main contributions. First, we show that for zero-sum bilinear games the average-case optimal method is the optimal method for the minimization of the Hamiltonian. Second, we provide an explicit expression for the optimal method corresponding to normal matrices, potentially non-symmetric. Finally, we specialize it to matrices with eigenvalues located in a disk and show a provable speed-up compared to worst-case optimal algorithms. We illustrate our findings through benchmarks with a varying degree of mismatch with our assumptions.
\end{abstract}

\section{Introduction}
The traditional analysis of optimization algorithms is a worst-case analysis \citep{nemirovski1995information,nesterov2004introductory}. This type of analysis provides a complexity bound for any input from a function class, no matter how unlikely. However, since hard-to-solve inputs might rarely occur in practice, the worst-case complexity bounds might not be representative of the observed running time.

A more representative analysis is given by the average-case complexity, averaging the algorithm’s complexity over all possible inputs. This analysis is standard for analyzing, e.g., sorting \citep{knuth1997theart} and cryptography algorithms \citep{katz2014introduction}. Recently, a line of work \citep{berthier2020accelerated,pedregosa2020average, lacotte2020optimal,paquette2020halting} focused on optimal methods for the optimization of quadratics, specified by a symmetric matrix. While worst-case analysis uses bounds on the matrix eigenvalues to yield upper and lower bounds on convergence, average-case analysis relies on the expected distribution of eigenvalues and provides algorithms with sharp optimal convergence rates. While the algorithms developed in this context have been shown to be efficient for minimization problems, these have not been extended to smooth games.

A different line of work considers smooth games but studies \emph{worst-case} optimal methods \citep{azizian2020accelerating}.
In this work, we combine the two previous trends and develop novel average-case optimal algorithms for finding the root of a linear system determined by a (potentially non-symmetric) normal matrix. We make the following {\bfseries main contributions}:
\begin{itemize}[leftmargin=*]
    \item Inspired by the problem of finding equilibria in smooth games, we develop average-case optimal algorithms for finding the root of a non-symmetric affine operator, both under a normality assumption (Thm. \ref{thm:opt_alg}), and under the extra assumption that eigenvalues of the operator are supported in a disk (Thm. \ref{thm:circular}). The proposed method, and its asymptotic variant, show a polynomial speedup compared to worst-case optimal method, verified by numerical simulations.
    \item We make a novel connection between average-case optimal methods for optimization, and average-case optimal methods for bilinear games. In particular, we show that solving the Hamiltonian using an average-case optimal method is optimal (\autoref{thm:bilinear_thm1}). This result complements  \citep{azizian2020accelerating}, who proved that Polyak Heavy Ball algorithm on the Hamiltonian is asymptotically worst-case optimal.
\end{itemize}

\section{Average-case analysis for normal matrices} \label{sec:general_sec}

In this paper we consider the following class of problems.
\begin{definition} \label{def:class_opt}
Let $\AA \in \RR^{d \times d}$ be a real matrix and $\xx^\star \in \RR^d$ a vector. The  non-symmetric (affine) operator (\textbf{NSO}) problem is defined as:
\begin{empheq}{equation*}\tag{NSO}\label{eq:nonsymop}
    \text{Find } \xx \; : \;  F(\xx) \defas \AA(\xx\!-\!\xx^{\star}) = {\boldsymbol 0}\,.
\end{empheq}
\end{definition}

This problem generalizes that of minimization of a convex quadratic function $f$, since we can cast the latter in this framework by setting the operator $F = \nabla f$. 
The set of solutions is an affine subspace that we will denote $\mathcal{X}^\star$. We will find convenient to consider the distance to this set, defined as
\begin{equation} \label{eq:def_distance}
    \text{dist}(\xx,\, \mathcal{X}^\star) \defas \min_{\vv\in \mathcal{X}^\star} \|\xx-\vv\|^2,\;\; \text{ with } \mathcal{X}^\star = \{ \xx \in \mathbb{R}^{d}\, | \,\AA (\xx -\xx^\star) = \mathbf{0} \}\,.
\end{equation}

In this paper we will develop \emph{average-case} optimal methods. For this, we consider $\AA$ and $\xx^\star$ to be random vectors, and a random initialization $\xx_0$. This induces a probability distribution over \ref{eq:nonsymop} problems, and we seek to find methods that have an optimal \emph{expected} suboptimality w.r.t. this distribution. More precisely, average-case optimal methods solve the following at each iteration $t$:
\begin{empheq}[box=\mybluebox]{equation}\label{eq:optimal_algo_problem}
  \vphantom{\sum}\min_{\xx_t}\mathbb{E}_{(\AA,\xx^{\star},\xx_0)} \dist(\xx_t,\, \mathcal{X}^\star) \quad \text{s.t.} \;\; \xx_{i} \in \xx_0 + \Span(\{ F(\xx_j)\}_{j=0}^{i-1}), \; \forall i \in [1:t].
\end{empheq}

The last condition on $\xx_t$ stems from restricting the class of algorithms to first-order methods.
This class encompasses many known schemes such as gradient descent with momentum, or full-matrix AdaGrad. However, methods such as Adam \citep{kingma_adam_2014} or diagonal AdaGrad \citep{duchi_adaptive_2011} are \textit{not} in this class, as the diagonal re-scaling creates iterates $\xx_t$ outside the span of previous gradients.
Although we will focus on the distance to the solution, the results can be extended to other convergence criteria such as $\|F(\xx_t)\|^2$.

Finally, note that the expectations in this paper are on the problem instance and \textit{not} on the randomness of the algorithm.

\subsection{Orthogonal residual polynomials and first-order methods}

The analysis of first-order methods simplifies through the use of polynomials. This section provides the tools required to leverage this connection.

\begin{definition} \label{def:residual_polynomials}
A \textbf{residual polynomial} is a polynomial $P$ that satisfies $P(0)=1$.
\end{definition}

\begin{restatable}{prop}{linkalgopolynomial}
    \label{prop:link_algo_polynomial} \citep{hestenes1952methods}
    If the sequence $(\xx_t)_{t \in \mathbb{Z}_+}$ is generated by a first-order method, then there exist residual polynomials $P_t$, each one of degree at most $t$, verifying
    \begin{equation}\label{eq:polynomial_iterates}
        \xx_{t}-\xx^\star = P_t(\AA)(\xx_0-\xx^\star)~ \quad \forall \, i\in\{0,\ldots,t\}\,.
    \end{equation}
\end{restatable}
As we will see, optimal average-case method are strongly related to orthogonal polynomials. We first define the inner product between polynomials.
\begin{definition}\label{def:scalar_product_poly}
    For $P, Q \in \mathbb{R}[X]$, we define the inner product $\langle \cdot, \cdot \rangle_\mu$ for a measure $\mu$ over $\mathbb{C}$ as
    \begin{align} \label{eq:scalar_product}
        \langle P, Q \rangle_{\mu} \defas \int_{\mathbb{C}} P(\lambda) Q(\lambda)^* \dif\mu(\lambda)\,.
\end{align}
\end{definition}

\begin{definition} \label{def:orthogonal_poly}
A sequence of polynomials $\{P_i\}$ is \textbf{orthogonal} (resp. \textbf{orthonormal}) w.r.t. $\langle \cdot, \cdot \rangle_{\mu}$ if
\[
    \langle P_i, P_i \rangle_{\mu} > 0 \;\; \text{(resp. $= 1$)}; \qquad \langle P_i, P_j \rangle_{\mu} = 0 \;\; \text{if $i\neq j$}.
\]
\end{definition}

\subsection{Expected Spectral Distribution}

Following \citep{pedregosa2020average}, we make the following assumption on the problem family.
\begin{assumption} \label{ass:assumption_x0_xstar}
    $\xx_0-\xx^\star$ is independent of $\AA$, and $\mathbb{E}_{\xx_0,\,\xx^\star}[(\xx_0-\xx^\star)(\xx_0-\xx^\star)^\top] = \frac{R^2}{d} \II_d$.
\end{assumption}

We will also require the following definitions to characterize difficulty of a problem class. Let $\{\lambda_1, \ldots, \lambda_d\}$ be the eigenvalues of a matrix $\AA\in\mathbb{R}^{d\times d}$. We define the \textbf{empirical spectral distribution} of $\AA$ as the probability measure
\begin{equation}\label{eq:empirical_spectral_density}
    \mu_{\AA}(\lambda) \defas {\textstyle{\frac{1}{d}\sum_{i=1}^d}} \delta_{\lambda_i}(\lambda)\,,
\end{equation}
where $\delta_{\lambda_i}$ is the Dirac delta, a distribution equal to zero everywhere except at $\lambda_i$ and whose integral over the entire real line is equal to one. Note that with this definition, $\int_{\mathcal{D}} \dif \mu_{\AA}(\lambda)$ corresponds to the proportion of eigenvalues in $\mathcal{D}$.

When $\AA$ is a matrix-valued random variable, $\mu_{\AA}$ is a measure-valued random variable. As such, we can define its \textbf{expected spectral distribution} 
\begin{equation}
    \mu \defas \EE_{\AA}[\mu_{\AA}]\,,
\end{equation}
which by the Riesz representation theorem is the measure that verifies $\int f \dif \mu = \EE_{\AA}[\int f \dif\mu_{\AA}]$ for all measureable $f$. Surprisingly, the expected spectral distribution is the only required characteristic to design optimal algorithms in the average-case.

\subsection{Expected error of first-order methods}

In this section we provide an expression for the expected convergence in terms of the residual polynomial and the expected spectral distribution introduced in the previous section. To go further in the analysis, we have to assume that $\AA$ is a normal matrix.
\begin{assumption}\label{assump:normal}
The (real) random matrix $\AA$ is normal, that is, it verifies $\AA\AA^\top = \AA^\top \AA$. 
\end{assumption}
Normality is equivalent to $\AA$ having the spectral decomposition
$\AA = \UU{\boldsymbol\Lambda} \UU^*$, where $\UU$ is unitary, i.e., $\UU^*\UU=\UU\UU^*=\textbf{I}$. We now have everything to write the expected error of a first-order algorithm applied to (\ref{eq:nonsymop}).

\begin{framed}
\vspace{-1ex}
\begin{restatable}{thm}{expectation} \label{thm:expectation}
Consider the application of a first-order method associated to the sequence of polynomials $\{P_t\}$ (\autoref{prop:link_algo_polynomial}) on the problem (\ref{eq:nonsymop}). Let $\mu$ being the spectral distribution of $\AA$. Under Assumptions \ref{ass:assumption_x0_xstar} and \ref{assump:normal}, we have
\begin{align} \label{eq:expect_dist}
    \mathbb{E} [\dist(\xx_t,\mathcal{X}^\star)] = R^2 \int_{\mathbb{C} \setminus \{0\}} |P_t|^2 \dif\mu\,,
\end{align}
\vspace{-3ex}
\end{restatable}
\end{framed}

Before designing optimal algorithms for certain specific distributions, we compare our setting with the average-case accelerating for minimization problems of \citet{pedregosa2020average}, who proposed optimal \textit{optimization} algorithms in the average-case.

\subsection{Difficulties of First-Order Methods on Games and Related Work} \label{sec:difficulties}

This section compares our contribution with the existing framework of average-case optimal methods for quadratic minimization problems.

\begin{definition} Let $\HH \in \RR^{d \times d}$ be a random symmetric positive-definite matrix and $\xx^\star \in \RR^d$ a random vector. These elements determine the following \textbf{random quadratic minimization problem}
\begin{empheq}{equation*}\tag{OPT}\label{eq:quad_optim}
  \textstyle \min_{\xx \in \RR^d} \big\{ f(\xx) \defas\!\mfrac{1}{2}(\xx\!-\!\xx^\star)^\top\!\HH(\xx\!-\!\xx^\star) \big\}\,.
\end{empheq}
\end{definition}

As in our paper, \citet{pedregosa2020average} find deterministic optimal first-order algorithms in \textit{expectation} w.r.t. the matrix $\HH$, the solution $\xx^{\star}$, and the initialization $\xx_0$. Since they work with problem \eqref{eq:quad_optim}, their problem is equivalent to \eqref{eq:nonsymop} with the matrix $\AA=\HH$. However, they have the \textit{stronger} assumption that the matrix is \textit{symmetric}, which implies being normal. The normality assumption is restrictive in the case of game theory, as they do not always naturally fit such applications. However, this set is expressive enough to consider interesting cases, such as bilinear games, and our experiments show that our findings are also consistent with non-normal matrices.


Using orthogonal residual polynomials and spectral distributions, they derive the explicit formula of the expected error. Their result is similar to \autoref{thm:expectation}, but the major difference is the domain of the integral, a real positive line in convex optimization, but a shape in the complex plane in our case. This shape plays a crucial role in the rate of converge of first-order algorithms, as depicted in the work of \citet{azizian2020accelerating,bollapragada2018nonlinear}. 




In the case of optimization methods, they show that optimal schemes in the average-case follow a simple three-term recurrence arising from the three-term recurrence for residual orthogonal polynomials for the measure $\lambda \mu(\lambda)$. Indeed, by \autoref{thm:expectation} the optimal method corresponds to the residual polynomials minimizing $\langle P, P \rangle_{\mu}$, and the following result holds:

\begin{restatable}{thm}{recurenceorthogonalpolynomials}\label{thm:recurence_orthogonal_polynomials} \citep[\S 2.4]{fischer1996polynomial}
When $\mu$ is supported in the real line, the residual polynomial of degree $t$ minimizing $\langle P, P \rangle_{\mu}$ is given by the degree $t$ residual
orthogonal polynomial w.r.t. $\lambda \mu(\lambda)$.
\end{restatable}

However, the analogous result does not hold for general measures in $\mathbb{C}$, and hence our arguments will make use of the following \autoref{thm:assche} instead, which links the residual polynomial of degree at most $t$ that minimizes $\langle P, P \rangle_{\mu}$ to the sequence of orthonormal polynomials for $\mu$.

\begin{restatable}{thm}{assche}[Theorem 1.4 of \cite{assche97orthogonal}] \label{thm:assche}
Let $\mu$ be a positive Borel measure in the complex plane. The minimum of the integral $\int_{\mathbb{C}} |P(\lambda)|^2 \dif\mu(\lambda)$ over residual polynomials $P$ of degree lower or equal than $t$ is uniquely attained by the polynomial
\begin{align}
    P^{\star}(\lambda) = \frac{\sum_{k=0}^t \phi_k(\lambda) \phi_k(0)^*}{\sum_{k=0}^t |\phi_k(0)|^2}, 
    \quad 
    \text{ with optimal value }
    \int_{\mathbb{C}} |P^{\star}(\lambda)|^2 \dif\mu(\lambda) = \frac{1}{\sum_{k=0}^t |\phi_k(0)|^2}\,,
\end{align}
where $(\phi_k)_k$ is the orthonormal sequence of polynomials with respect to the inner product $\langle\cdot,\cdot\rangle_{\mu}$.
\end{restatable}

In the next sections we consider cases where the optimal scheme is identifiable.

\section{Average-case Optimal Methods for Bilinear Games}

We consider the problem of finding a Nash equilibrium of the zero-sum minimax game given by
\begin{equation}
    \min_{\thetaa_1}\max_{\thetaa_2}\mathbf{\ell}(\thetaa^{\vphantom{T}}_1,\,\thetaa^{\vphantom{T}}_2)\defas (\thetaa_1 - \thetaa_1^\star)^\top \MM (\thetaa_2^{\vphantom{T}} - \thetaa_2^\star)\,.
\end{equation}
Let $\thetaa_1, \thetaa_1^\star \in \mathbb{R}^{d_1}, \thetaa_2, \thetaa_2^\star \in \mathbb{R}^{d_2}, \MM \in \mathbb{R}^{d_1 \times d_2}$ and $d \defas d_1 + d_2$. The vector field of the game \citep{mechanics2018balduzzi} is defined as $F(\xx) = \AA(\xx-\xx^{\star})$, where
\begin{align} \label{eq:bilinear_def_A}
    F(\thetaa_1,\,\thetaa_2) = 
    \begin{bmatrix}
        \hphantom{-}\nabla_{\thetaa_1} \ell(\thetaa_1,\,\thetaa_2) \\
        -\nabla_{\thetaa_2} \ell(\thetaa_1,\,\thetaa_2)
    \end{bmatrix}
    = \underbrace{\begin{bmatrix}
    0 & \MM \\
    -\MM^\top & 0
    \end{bmatrix}}_{=\AA}
    \Bigg( \underbrace{\begin{bmatrix}
        \thetaa_1 \\ 
        \thetaa_2
    \end{bmatrix}}_{=\xx} - \underbrace{\begin{bmatrix}
     \thetaa_1^\star \\ 
     \thetaa_2^\star
    \end{bmatrix}}_{=\xx^\star}
    \Bigg) = \AA(\xx-\xx^{\star})\,.
\end{align}
As before, $\mathcal{X}^{\star}$ denotes the set of points $\xx$ such that $F(\xx) = 0$, which is equivalent to the set of Nash equilibrium. If $\MM$ is sampled independently from $\xx_0, \xx^\star$ and $\xx_0 - \xx^\star$ has covariance $\frac{R^2}{d} \II_d$, \autoref{ass:assumption_x0_xstar} is fulfilled. Since $\AA$ is skew-symmetric, it is in particular normal and \autoref{assump:normal} is also satisfied.

We now show that the optimal average-case algorithm to solve bilinear problems is Hamiltonian gradient descent with momentum, described below in its general form. Contrary to the methods in \citet{azizian2020accelerating}, the method we propose is \textit{anytime} (and not only asymptotically) average-case optimal.

\begin{framed}
\vspace{-0.5ex}
\textbf{Optimal average-case algorithm for bilinear games.}\\
~ \\
\textbf{Initialization.} $\xx_{-1} =\xx_0 = \big(\thetaa_{1,0}, \; \thetaa_{2,0}\big)$, sequence $\{h_t,m_t\}$ given by \autoref{thm:bilinear_thm1}. \\
\textbf{Main loop.} For $t\geq 0,$
\begin{align}
\begin{split}
    \gg_t & \textstyle  = F(\xx_{t} - F(\xx_t)) - F(\xx_t) \qquad\quad \left(= \frac{1}{2} \nabla  \|F(\xx_t)\|^2 \; \text{ by \eqref{eq:relation_grad_extragrad}}\right)\\
    \xx_{t+1} & = \xx_{t} - h_{t+1}\gg_t +m_{t+1} (\xx_{t-1} - \xx_{t})  
\end{split}\label{eq:opti_scheme_bilinear}
\end{align}
\end{framed}
The quantity $\frac{1}{2} \|F(\xx)\|^2$ is commonly known as the Hamiltonian of the game \citep{mechanics2018balduzzi}, hence the name \textit{Hamiltonian gradient descent}. Indeed, $\gg_t = \nabla \left(\frac{1}{2} \|F(\xx)\|^2\right)$ when $F$ is affine: 
\begin{align}
\begin{split}
    F(\xx - F(\xx)) - F(\xx) &= \AA(\xx - \AA (\xx - \xx^{\star}) - \xx^{\star}) - \AA (\xx - \xx^{\star}) = - \AA(\AA (\xx - \xx^{\star})) \\ 
    &= \AA^\top (\AA (\xx - \xx^{\star})) = \nabla \left(\frac{1}{2} \|\AA (\xx - \xx^{\star})\|^2 \right) = \nabla \left(\frac{1}{2} \|F(\xx)\|^2 \right)\,. \label{eq:relation_grad_extragrad}
\end{split}
\end{align}
The following theorem shows that \eqref{eq:opti_scheme_bilinear} is indeeed the optimal average-case method associated to the minimization problem $\min_\xx \left(\frac{1}{2} \|F(\xx)\|^2\right)$, as the following theorem shows.

\begin{framed}
\vspace{-1ex}
\begin{restatable}{thm}{bilinearthmone} \label{thm:bilinear_thm1}
Suppose that \autoref{ass:assumption_x0_xstar} holds and that the spectral distribution of $\MM \MM^\top$ is absolutely continuous with respect to the Lebesgue measure. 
Then, 
the method \eqref{eq:opti_scheme_bilinear} is average-case optimal for bilinear games 
when $h_t,\,m_t$ are chosen to be the coefficients of the average-case optimal minimization of $\frac{1}{2} \|F(\xx)\|^2$. 
\end{restatable}
\vspace{-1ex}
\end{framed}
\textbf{How to find optimal coefficients?} Since $\frac{1}{2} \|F(\xx)\|^2$ is a quadratic problem, the coefficients $\{h_t,\,m_t\}$ can be found using the average-case framework for quadratic minimization problems of \citep[Theorem 3.1]{pedregosa2020average}. 

\textit{Proof sketch.} When computing the optimal polynomial $\xx_{t} = P_t(\AA)(\xx_0-\xx^{\star})$, we have that the residual orthogonal polynomial $P_t$ behaves differently if $t$ is even or odd.
\begin{itemize}[leftmargin=*]
    \item \textbf{Case 1: $t$ is even.} In this case, we observe that the polynomial $P_t(\AA)$ can be expressed as $Q_{t/2}(-\AA^2)$, where $(Q_t)_{t \geq 0}$ is the sequence of orthogonal polynomials w.r.t. the expected spectral density of $-\AA^2$, whose eigenvalues are real and positive. This gives the recursion in \eqref{eq:opti_scheme_bilinear}.
    \item \textbf{Case 2: $t$ is odd.} There is no residual orthogonal polynomial of degree $t$ for $t$ odd. Instead, odd iterations do correspond to the intermediate computation of $\gg_t$ in \eqref{eq:opti_scheme_bilinear}, but not to an actual iterate.
\end{itemize}

\subsection{Particular case: \texorpdfstring{$\mathbf{M}$}{M} with i.i.d. components}

We now show the optimal method when the entries of $\MM$ are i.i.d. sampled. For simplicity, we order the players such that $d_1 \leq d_2$.

\begin{assumption} \label{ass:marchenko}
Assume that each component of $\MM$ is sampled iid from a distribution of mean 0 and variance $\sigma^2$, and we take $d_1, d_2 \rightarrow \infty$ with $\frac{d_1}{d_2} \rightarrow r < 1$.
\end{assumption}

In such case, the spectral distribution of $\frac{1}{d_2}\MM \MM^\top$ tends to the Marchenko-Pastur law, supported in $[\ell,L]$ and with density:
\begin{align}
    \rho_{MP}(\lambda) \defas \frac{\sqrt{(L-\lambda)(\lambda-\ell)}}{2\pi \sigma^2 r \lambda}, \quad \text{where } L \defas \sigma^2 (1+\sqrt{r})^2, \ell \defas \sigma^2 (1-\sqrt{r})^2.
\end{align}


\begin{restatable}{prop}{bilinearthm2} \label{thm:bilinear_thm2}
When $\MM$ satisfies \autoref{ass:marchenko}, the optimal parameter of scheme \eqref{eq:opti_scheme_bilinear} are
\begin{align}
\begin{split} \label{eq:opti_scheme}
    \textstyle h_t = -\frac{\delta_t}{\sigma^2 \sqrt{r}}, \;\; m_t = 1+\rho\delta_t,\quad \text{where} \; \; \rho = \frac{1 + r}{\sqrt{r}}, \;\; \delta_t = (-\rho-\delta_{t-1})^{-1}, \;\; \delta_0 = 0.
\end{split}
\end{align}
\end{restatable}

\begin{proof}
By \autoref{thm:bilinear_thm1}, the problem reduces to finding the optimal average-case algorithm for the problem $\min_{\xx}\frac{1}{2} \|F(\xx)\|^2$. Since the expected spectral distribution of $\frac{1}{d_2} \MM \MM^\top$ is the Marchenko-Pastur law, we can use the optimal algorithm from \citep[Section 5]{pedregosa2020average}.
\end{proof}



\section{General average-case optimal method for normal operators}


In this section we derive general average-case optimal first-order methods for normal operators. First, we need to assume the existence of a three-term recurrence for residual orthogonal polynomials (\autoref{ass:three_term_residual}). As mentioned in \autoref{sec:difficulties}, for general measures in the complex plane, the existence of a three-term recurrence of orthogonal polynomials is not ensured. In \autoref{prop:3termcond} in \autoref{sec:appendix_b} we give a sufficient condition for its existence, and in the next subsection we will show specific examples where the residual orthogonal polynomials satisfy the three-term recurrence. 


\begin{assumption}[Simplifying assumption] \label{ass:three_term_residual}
The sequence of residual polynomials $\{\psi_t \}_{t \geq 0}$ orthogonal w.r.t. the measure $\mu$, defined on the complex plane, admits the three-term recurrence 
\begin{align}
    \begin{split} \label{eq:three_term_residual_eq}
        \psi_{-1} = 0, \quad \psi_{0} = 1, \quad \psi_t(\lambda) = (a_t + b_t \lambda) \psi_{t-1}(\lambda) + (1-a_t) \psi_{t-2}(\lambda).
    \end{split}
\end{align}
\end{assumption}


Under \autoref{ass:three_term_residual}, \autoref{thm:opt_alg} shows that the optimal algorithm can also be written as an average of iterates following a simple three-terms recurrence.

\begin{framed}
\vspace{-2ex}
\begin{restatable}{thm}{optalg} \label{thm:opt_alg}
Under \autoref{ass:three_term_residual} and the assumptions of \autoref{thm:expectation}, the following algorithm is optimal in the average case, with $\yy_{-1}=\yy_0 = \xx_0$:
\begin{align}
    \yy_t & = a_t \yy_{t-1} + (1- a_t) \yy_{t-2} + b_t F(\yy_{t-1})\nonumber \\
    \xx_t &=  \frac{B_t }{B_t + \beta_t} \xx_{t-1} + \frac{\beta_t}{B_t+\beta_t} \yy_t\, , \quad \beta_t = \phi_t^2(0), \quad B_t = B_{t-1} + \beta_{t-1}, \quad B_0 = 0\,. \vspace{-3ex} \label{eq:opti_algorithm_disk}
\end{align}
where $(\phi_{k}(0))_{k \geq 0}$ can be computed using the three-term recurrence (upon normalization). Moreover, $\EE_{(\AA,\xx^{\star},\xx_0)} \dist(\xx_t,\mathcal{X}^\star)$ converges to zero at rate ${1}/{B_t}$.
\end{restatable}
\vspace{-2ex}
\end{framed}
\textbf{Remark.} \quad Notice that it is not immediate that \eqref{eq:opti_algorithm_disk} fulfills the definition of first-order algorithms stated in \eqref{eq:optimal_algo_problem}, as $\yy_t$ is clearly a first-order method but $\xx_t$ is an average of the iterates $\yy_t$. Using that $F$ is an affine function we see that $\xx_t$ indeed fulfills \eqref{eq:optimal_algo_problem}. 

\textbf{Remark.} \quad \autoref{ass:three_term_residual} is needed for the sequence $(\yy_t)_{t \geq 0}$ to be computable using a three-term recurrence. However, for some distribution, the associated sequence of orthogonal polynomials may admit another recurrence that may not satisfy \autoref{ass:three_term_residual}.

\subsection{Circular spectral distributions}
In random matrix theory, the circular law states that if $\AA$ is an $n \times n$ matrix with i.i.d. entries of mean $C$ and variance $R^2/n$, as $n \rightarrow \infty$ the spectral distribution of $\AA$ tends to the uniform distribution on $D_{C,R}$. In this subsection we apply \autoref{thm:opt_alg} to a class of spectral distributions specified by \autoref{ass:circular_measure}, which includes the uniform distribution on $D_{C,R}$. Even though the random matrices with i.i.d entries are not normal, in \autoref{sec:experiments} we see that the empirical results for such matrices are consistent with our theoretical results under the normality assumption.

\begin{assumption} \label{ass:circular_measure}
Assume that the spectral distribution $\mu_{\AA}$ is supported in the complex plane on the disk $D_{C,R}$ of center $C \in \mathbb{R}, C > 0$ and radius $R < C$. Moreover, assume that the spectral density is circularly symmetric, i.e. there exists a probability measure $\mu_R$ supported on $[0,R]$ such for all $f$ measurable and $r \in [0,R]$, $\dif \mu_{\AA}(C + r e^{i\theta}) = \frac{1}{2\pi} \dif \theta \dif \mu_R(r)$.
\end{assumption}

\begin{restatable}{prop}{orthocircular}
    \label{prop:ortho_circular}
    If $\mu$ satisfies \autoref{ass:circular_measure}, the sequence of orthonormal polynomials is $(\phi_t)_{t\geq 0}$,
    \begin{align}
        \textstyle \phi_t(\lambda) = \dfrac{(\lambda-C)^t}{K_{t,R}}, \text{ where } K_{t,R} = \sqrt{\int_{0}^R r^{2t} \dif \mu_{R}(r)}\,.
    \end{align}
\end{restatable}
\paragraph{Example. }
The uniform distribution in $D_{C,R}$ is to $\dif \mu_R = \frac{2r}{R^2} \dif r$, and $K_{t,R} = R^t/\sqrt{t+1}$.

From \autoref{prop:ortho_circular}, the sequence of residual polynomials is given by $\phi_t(\lambda)/\phi_t(0) = \left(1 - \frac{\lambda}{C}\right)^t$, which implies that \autoref{ass:three_term_residual} is fulfilled with $a_t = 1, b_t = - \frac{1}{C}$. Thus, by \autoref{thm:opt_alg} we have
\begin{framed}
\vspace{-2ex}
\begin{restatable}{thm}{circularthmone} \label{thm:circular}
Given an initialization $\xx_0 (\yy_0 = \xx_0)$, if \autoref{ass:circular_measure} is fulfilled with $R < C$ and the assumptions of \autoref{thm:expectation} hold, then the average-case optimal first-order method is
\begin{align}
\begin{split} \label{eq:opti_algorithm_disk2}
    \yy_t &= \textstyle \yy_{t-1} -\frac{1}{C} F(\yy_{t-1}), \quad \beta_t = C^{2t}/K_{t,R}^{2}, \quad B_t = B_{t-1} + \beta_{t-1},\\
    \xx_t &=  \frac{B_t}{B_t+\beta_t} \xx_{t-1} + \frac{\beta_t}{B_t+\beta_t} \yy_t.
\end{split}
\end{align}
Moreover, $\EE_{(\AA,\xx^{\star},\xx_0)} \dist(\xx_t,\mathcal{X}^\star)$ converges to zero at rate ${1}/{B_t}$.
\end{restatable}
\vspace{-2ex}
\end{framed}
We now compare \autoref{thm:circular} with worst-case methods studied in \cite{azizian2020accelerating}. They give a worst-case convergence lower bound of $(R/C)^{2t}$ on the quantity $\dist(\zz_t,\mathcal{X}^\star)$ for first-order methods $(\zz_t)_{t \geq 0}$ on matrices with eigenvalues in the disk $D_{C,R}$. By the classical analysis of first-order methods, this rate is achievable by gradient descent with stepsize $1/C$, i.e. the iterates $\yy_t$ defined in \eqref{eq:opti_algorithm_disk2}. However, by equation \eqref{eq:limiting_rates} in  \autoref{prop:convrates} we have that under slight additional assumptions (those of \autoref{thm:asymptotic2}), $\lim_{t \rightarrow \infty} \EE [\dist(\xx_t,\mathcal{X}^\star)]/\EE [\dist(\yy_t,\mathcal{X}^\star)] = 1 - \frac{R^2}{C^2}$ holds. That is, the average-case optimal algorithm outperforms gradient descent by a constant factor depending on the conditioning $R/C$, showcasing that average-case analysis is subtler than worst-case analysis.   




\section{Asymptotic behavior}
The recurrence coefficients of the average-case optimal method typically converges to limiting values when $t\rightarrow \infty$, which gives an "average-case asymptotically optimal first-order method" with constant coefficients. For the case of symmetric operators with spectrum in $[\ell,L]$, \citet{scieur2020universal} show that under mild conditions, the asymptotically optimal algorithm is the Polyak momentum method with coefficients depending only on $\ell$ and $L$. For bilinear games, since the average-case optimal algorithm is the average-case optimal algorithm of an optimization algorithm, we can make use of their framework to obtain the asymptotic algorithm (see Theorem 3 of \citet{scieur2020universal}).

\begin{restatable}{prop}{asymptotic1} \label{thm:asymptotic1}
Assume that the spectral density $\mu_{\MM \MM^{\top}}$ of $\MM \MM^\top$ is supported in $[\ell, L]$ for $0 < \ell < L$, and strictly positive in this interval. Then, the asymptotically optimal algorithm for bilinear games is the following version of Polyak momentum:
\begin{align}
\begin{split} \label{eq:asympt_bilinear}
    \gg_t & = F(\xx_{t} - F(\xx_t)) - F(\xx_t)\\ 
    \xx_{t+1} &= \textstyle \xx_t + \left( \frac{\sqrt{L} - \sqrt{\ell}}{\sqrt{L} + \sqrt{\ell}} \right)^2 (\xx_{t-1} - \xx_t) - \left( \frac{2}{\sqrt{L} + \sqrt{\ell}} \right)^2 \gg_t
\end{split}
\end{align}
\end{restatable}
Notice that the algorithm in \eqref{eq:asympt_bilinear} is the worst-case optimal algorithm from Proposition 4 of \cite{azizian2020accelerating}.
For the case of circularly symmetric spectral densities with support on disks, we can also compute the asymptotically optimal algorithm. 
\begin{restatable}{prop}{asymptoticcirc} \label{thm:asymptotic2}
Suppose that the assumptions of \autoref{thm:circular} hold with $\mu_R \in \mathcal{P}([0,R])$ fulfilling $\mu_R([r,R]) = \Omega((R-r)^\kappa)$ for $r$ in $[r_0,R]$ for some $r_0 \in [0,R)$ and for some $\kappa \in \mathbb{Z}$. 
Then, the average-case asymptotically optimal algorithm is, with $\yy_0 = \xx_0$:
\begin{align}
\begin{split} \label{eq:opt_circ_asymp}
    \yy_t &= \textstyle \yy_{t-1} -\frac{1}{C} F(\yy_{t-1}), \\
    \xx_t &= \textstyle \left(\frac{R}{C}\right)^2 \xx_{t-1} + \left(1 - \left(\frac{R}{C}\right)^2 \right) \yy_t.
\end{split}
\end{align}
Moreover, the convergence rate for this algorithm is asymptotically the same one as for the optimal algorithm in \autoref{thm:circular}. Namely, $\lim_{t \rightarrow \infty} \EE[\dist(\xx_t,\mathcal{X}^\star)] B_t = 1$.
\end{restatable}
The condition on $\mu_R$ simply rules out cases in which the spectral density has exponentially small mass around 1. It is remarkable that in algorithm \eqref{eq:opt_circ_asymp} the averaging coefficients can be expressed so simply in terms of the quantity $R/C$. Notice also that while the convergence rate of the algorithm is slower than the convergence rate for the optimal algorithm by definition, both rates match in the limit, meaning that the asymptotically optimal algorithm also outperforms gradient descent by a constant factor $1-\frac{R^2}{C^2}$ in the limit $t \rightarrow \infty$.


\section{Experiments} \label{sec:experiments}

We compare some of the proposed methods on settings with varying degrees of mismatch with our assumptions.

\paragraph{Bilinear Games.} We consider min-max bilinear problems of the form \eqref{eq:bilinear_def_A}, where the entries of $\MM$ are generated i.i.d. from a standard Gaussian distribution. We vary the ratio $r=d/n$ parameter for $d=1000$ and compare the average-case optimal method of Theorems~\ref{thm:bilinear_thm1} and \ref{thm:asymptotic1}, the asymptotic worst-case optimal method of \citep{azizian2020accelerating} and extragradient~\citep{korpelevich1976extragradient}. In all cases, we use the convergence-rate optimal step-size assuming knowledge of the edges of the spectral distribution.

The spectral density for these problems is displayed in the first row of Figure~\ref{fig:experiments} and the benchmark results on the second row. Average-case optimal methods always outperform other methods, and the largest gain is in the ill-conditioned regime $(r \approx 1)$.

\paragraph{Circular Distribution.}
For our second experiment we choose $\AA$ as a matrix with iid Gaussian random entries, therefore the support of the distribution of its eigenvalue is a disk. Note that $\AA$ does not satisfy the normality assumption of Assumption~\ref{assump:normal}. Figure~\ref{fig:experiments} (third row) compares the average-case optimal methods from Theorems \ref{thm:circular} and \ref{thm:asymptotic2} on two datasets with different levels of conditioning. Note that the methods converge despite the violation of Assumption~\ref{assump:normal}, suggesting a broader applicability than the one proven in this paper. We leave this investigation for future work.

\begin{figure}
    \centering
    \includegraphics[width=\linewidth]{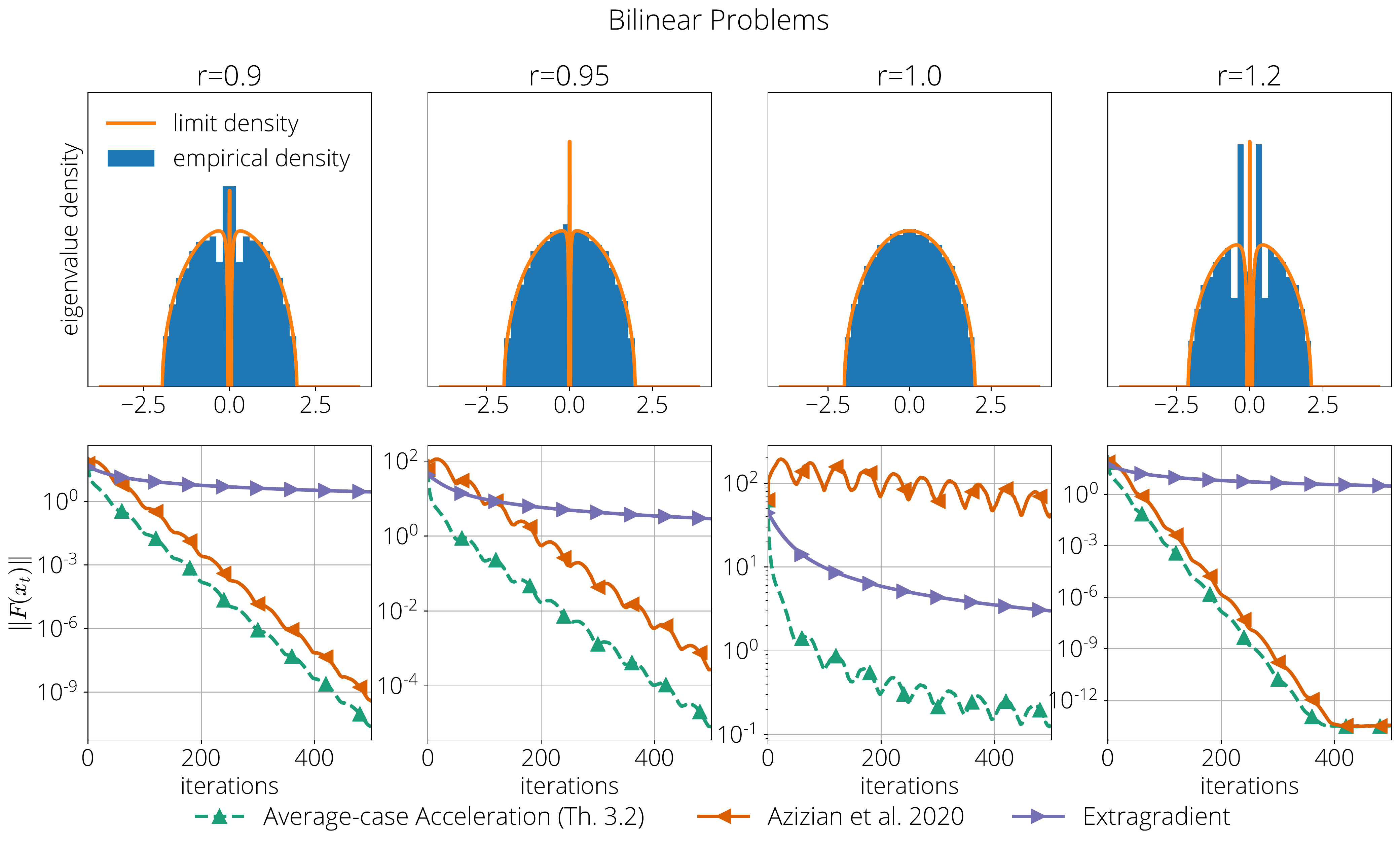}
    \includegraphics[width=0.97\linewidth]{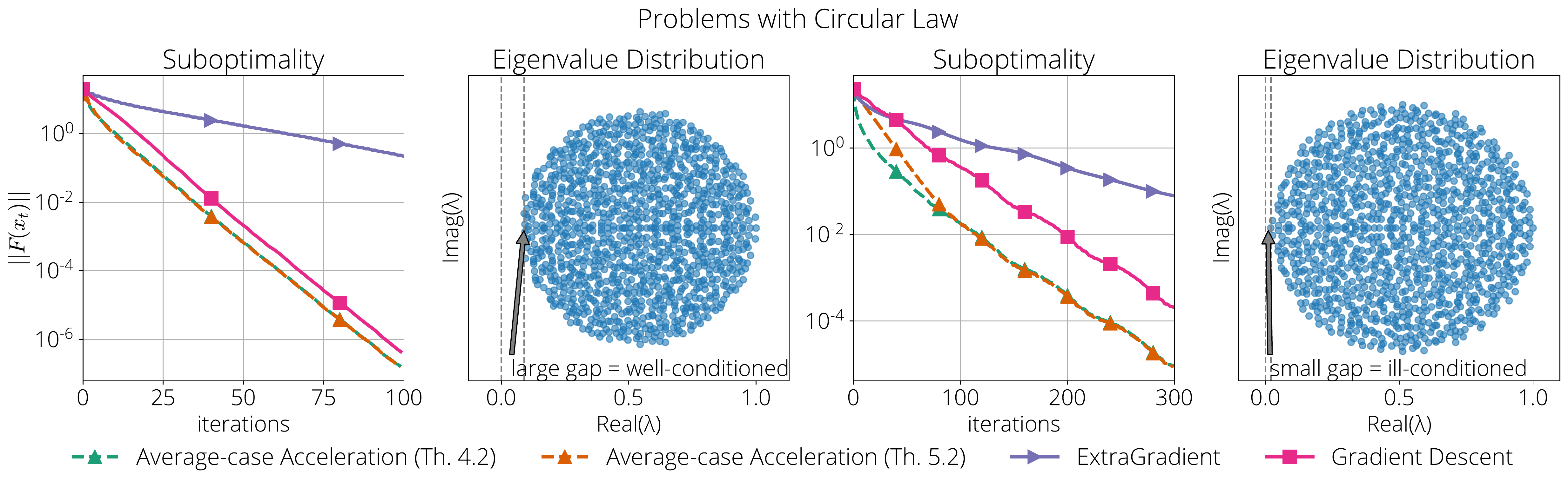}
    \caption{{\bfseries Benchmarks and spectral density for different games.} {\it Top row}: spectral density associated with bilinear games for varying decrees of the ratio parameter $r=n/d$. {\it Second row}: Benchmarks. Average-case optimal methods always outperform other methods, and the largest gain is in the ill-conditioned regime $(r \approx 1)$.
    {\it Third row}. Benchmarks (columns 1 and 3) and eigenvalue distribution of a design matrix generated with iid entries for two different degrees of conditioning. Depite the normality assumption not being satisfied, we still observe an improvement of average-case optimal methods vs worst-case optimal ones.}
    \label{fig:experiments}
\end{figure}

\section{Discussion and Future Research Directions}

In this paper, we presented a general framework for the design of optimal algorithms in the average-case for affine operators $F$, whose underlying matrix is possibly non-symmetric. However, our approach presents some limitations, the major one being the restriction to normal matrices. Fortunately, given our numerical experiments, it seems this assumption can be relaxed. As extensions, it would be interesting to analyze the nonlinear-case, as well as stochastic algorithms. Some recent works, such as \citep{loizou2020stochastic}, give some results in this direction in the worst-case setting.

\clearpage

\bibliography{biblio}

\clearpage

\appendix

\section{Proof of \autoref{thm:expectation}}

\subsection{Preliminaries}

Before proving \autoref{thm:expectation}, we quickly analyze the distance function \eqref{eq:def_distance}, recalled below,
\[
    \text{dist}(\xx,\, \mathcal{X}^\star) \defas \min_{\vv\in \mathcal{X}^\star} \|\xx-\vv\|^2.
\]
The definition of the distance function is not practical for the theoretical analysis. Fortunately, it is possible to find a simple expression that uses the orthogonal projection matrix $\Pi$ to the kernel $\text{Ker}(\AA)$. Since $\Pi$ is an orthogonal projection matrix to the kernel of a linear transformation, it satisfies
\begin{equation}
    \Pi = \Pi^T, \quad \Pi^2=\Pi, \quad \text{and}\;\; \AA\Pi  = 0. \label{eq:prop_projection}
\end{equation}
The normality assumption on $\AA$ implies also that
\begin{equation}
    \Pi \AA = 0. 
    \label{eq:prop_projection2}
\end{equation}
Indeed, the spectral decomposition of $\AA$ is
\[
    \AA = [\UU_1|\UU_2] 
    \begin{bmatrix}
    {\boldsymbol\Lambda} & 0 \\
    0 & 0
    \end{bmatrix}
    [\UU_1|\UU_2]^*,
\]
and then $\Pi = \UU_2 \UU_2^*$. The next proposition uses $\Pi$ to derive the explicit solution of the \eqref{eq:def_distance}.

\begin{restatable}{prop}{projection_formula}\label{pro:projection_formula}
We have that 
\begin{align*}
    \dist(\yy,\, \mathcal{X}^\star) & = \|(\II-\Pi)(\yy-\xx^\star)\|^2 \quad \forall \xx^{\star} \in \mathcal{X}^{\star}.
\end{align*}
\end{restatable}
\begin{proof}
    We first parametrize the set of solution $\mathcal{X}^\star$. By definition we have
    \[
        \mathcal{X}^\star = \{ \xx : \AA(\xx-\xx^{\star}) = 0 \}.
    \]
    Which can be written in terms of the kernel of $\AA$ as
    \[
        \mathcal{X}^\star = \{ \xx^\star + \Pi \ww : \ww \in \R^d \}.
    \]
    From this, we can rewrite the distance function \eqref{eq:def_distance} as 
    \[
        \text{dist}(\yy,\, \mathcal{X}^\star) = \min_{\ww \in \RR^d} \|\yy-(\xx^\star+\Pi\ww)\|^2.
    \]
    The minimum can be attained at different points, but in particular at $\ww = -(\yy-\xx^\star)$, which proves the statement.
\end{proof}

We now simplifies further the result of the previous proposition in the case where $\xx_{t}$ is generated by a first order method.
\begin{restatable}{prop}{projection_formula_first_order}\label{prop:projection_formula_first_order}
For every iterate $\xx_t$ of a first-order methods, i.e., $\xx_t$ satisfies
\[
    \xx_t-\xx^{\star} = P_t(\AA)(\xx_0-\xx^{\star}),\quad \deg(P_t) \leq t, \quad P(0) = \II,
\]
we have that 
\begin{align*}
    \dist(\xx_t,\, \mathcal{X}^\star) & = \|\xx_t-\xx^\star\|^2 - \|\Pi(\xx_0-\xx^\star)\|^2.
\end{align*}
\end{restatable}
\begin{proof}
We start with the result of \autoref{pro:projection_formula},
    \[
        \text{dist}(\xx_t,\, \mathcal{X}^\star) = \|(\II-\Pi)(\xx_t-\xx^\star)\|^2.
    \]
    The norm can be split into
    \begin{align*}
        \|(\II-\Pi)(\xx_t-\xx^\star)\|^2 & = \|\xx_t-\xx^\star\|^2 + \|\underbrace{\Pi^2}_{=\Pi \text{ by \eqref{eq:prop_projection}}}(\xx_t-\xx^\star)\|^2 - 2 \|\Pi (\xx_t-\xx^\star)\|^2\\
        & = \|\xx_t-\xx^\star\|^2 - \|\Pi(\xx_t-\xx^\star)\|^2.
    \end{align*}
    Since $\xx_t$ is generated by a first order method, we have
    \[
        \xx_t -\xx^\star = P_t(\AA)(\xx_0-\xx^\star),\quad P_t(0) = 1.
    \]
    Since $P(0)=1$, the polynomial can be factorized as $P(\AA) = \II + \AA\QQ_{t-1}(\AA)$,
    $\QQ_{t-1}$ being a polynomial of degree $t-1$. Therefore, $\left\|\Pi(\xx_t-\xx^\star)\right\|^2$ reads
    \begin{align*}
        \left\|\Pi(\xx_t-\xx^\star)\right\|^2 &= \left\|\Pi\left(\II + \AA\QQ_{t-1}(\AA)\right)(\xx_0-\xx^{\star})\right\|^2\\
        & = \|\Pi(\xx_0-\xx^{\star}) + \underbrace{\Pi\AA}_{=0 \text{ by \eqref{eq:prop_projection2}}}\QQ_{t-1}(\AA)(\xx_0-\xx^{\star})\|^2\\
        & = \left\|\Pi(\xx_0-\xx^{\star})\right\|^2,
    \end{align*}
    which prove the statement.
\end{proof}

\subsection{Proof of the theorem}

We are now ready to prove the main result.
\expectation*
\begin{proof}
    We start with the result of \autoref{prop:projection_formula_first_order},
    \[
        \text{dist}(\xx_t,\, \mathcal{X}^\star) = \|\xx_t-\xx^\star\|^2 - \|\Pi(\xx_0-\xx^\star)\|^2.
    \]
    We now write the expectation of the distance function,
    \begin{align*}
        \E[\text{dist}(\xx_t,\, \mathcal{X}^\star)] & = \E\left[ \|\xx_t-\xx^{\star}\|^2-\left\|\Pi(\xx_0-\xx^{\star})\right\|^2 \right] \\
        & = \E\left[ \|P_t(\AA)(\xx_0-\xx^\star)\|^2-\left\|\Pi(\xx_0-\xx^{\star})\right\|^2 \right] \\
        & = \E\left[ \tr P_t(\AA)P_t(\AA)^T\left(\xx_0-\xx^{\star}\right)\left(\xx_0-\xx^{\star}\right)^T-\tr\Pi^2(\xx_0-\xx^{\star})(\xx_0-\xx^{\star})^T \right]\\
        & =  \E_A\left[ \tr P_t(\AA)P_t(\AA)^T\E\left[\left(\xx_0-\xx^{\star}\right)\left(\xx_0-\xx^{\star}\right)^T|\AA\right]-\tr\Pi\E\left[(\xx_0-\xx^{\star})(\xx_0-\xx^{\star})^T |\AA\right]\right]\\
        & = R \E_A\left[ \tr P_t(\AA)P_t(\AA)^T - \tr\Pi \right]\\
        & = R \E\left[\sum_{i=1}^d |P(\lambda_i)|^2 - \tr\Pi\right]\\
        & = R \E\left[ \int_{\mathbb{C}\backslash \{0\}} |P(\lambda)|^2\delta_{\lambda_i}(\lambda) + |P(0)|^2\cdot\text{[\# zero eigenvalues]} - \tr\Pi\right]
    \end{align*}
    However, $|P(0)|^2=1$ and $\tr\Pi$ corresponds to the number of zero eigenvalues of $\AA$, therefore,
    \[
        E[\text{dist}(\xx_t,\, \mathcal{X}^\star)]  = R \E\left[ \int_{\mathbb{C}\backslash \{0\}} |P(\lambda)|^2\delta_{\lambda_i}(\lambda)\right] = R \int_{\mathbb{C}\backslash \{0\}} P(\lambda) \mu(\lambda).
    \]
\end{proof}

\section{Proofs of \autoref{thm:bilinear_thm1} and \autoref{thm:bilinear_thm2}} \label{sec:appendix_b}

\begin{restatable}{prop}{blockdeterminant}
    \label{prop:block_determinant} [Block determinant formula]
If $A,B,C,\DD$ are (not necessarily square) matrices, 
\begin{align}
    \text{det} \begin{bmatrix}
    \AA & \BB \\ \CC & \DD
    \end{bmatrix} = \text{det}(\DD) \text{det}(\AA - \BB \DD^{-1} \CC),
\end{align}
if $D$ is invertible.
\end{restatable}

\begin{definition}[Pushforward of a measure] \label{def:pushforward}
Recall that the pushforward $f_* \mu$ of a measure $\mu$ by a function $f$ is defined as the measure such that for all measurable $g$,
\begin{align}
    \int g(\lambda) \dif(f_* \mu)(\lambda) = \int g(f(\lambda)) \dif\mu(\lambda).
\end{align}
Equivalently, if $X$ is a random variable with distribution $\mu$, then $f(X)$ has distribution $f_* \mu$.
\end{definition}

\begin{restatable}{prop}{eigenvaluedist} \label{prop:eigenvalue_dist}
Assume that the dimensions of $\MM \in \mathbb{R}^{d_x \times d_y}$ fulfill $d_x \leq d_y$ and let $r=d_x/d_y$. Let $\mu_{\MM \MM^\top}$ be the spectral distribution of the random matrix $\MM \MM^\top \in \mathbb{R}^{d_x \times d_x}$, and assume that it is absolutely continuous with respect to the Lebesgue measure. The spectral distribution of
$\AA$ is contained in the imaginary line and is given by
\begin{align} \label{eq:mu_A}
    \mu_{\AA}(i\lambda) = \left( 1 -\frac{2}{1+\frac{1}{r}} \right) \delta_0(\lambda) + \frac{2|\lambda|}{1+\frac{1}{r}} \mu_{\MM \MM^\top}(\lambda^2)\,.
\end{align}
for $\lambda \in \mathbb{R}$. If $d_x \geq d_y$, then \eqref{eq:mu_A} holds with $\mu_{\MM^\top \MM}$ in place of $\mu_{\MM \MM^\top}$ and $1/r$ in place of $r$.
\end{restatable}
\begin{proof}
By the block determinant formula, we have that for $s \neq 0$,
\begin{align}
\begin{split}
    \text{det} \left(s \II_{d_1 + d_2} - 
    \AA \right) &= \begin{vmatrix}
    s \II_{d_1} & - \MM\\
    \MM^\top & s \II_{d_2}
    \end{vmatrix} = \text{det}(s \II_{d_2}) \text{det}(s \II_{d_1} + \MM s^{-1} \II_{d_2} \MM^\top) \\ &= s^{d_2-d_1} \text{det}(s^2 \II_{d_1} + \MM \MM^\top)
\end{split}
\end{align}
Thus, for every eigenvalue $-\lambda \leq 0$ of $-\MM \MM^\top$, both $i \sqrt{\lambda}$ and $-i \sqrt{\lambda}$ are eigenvalues of $\AA$. Since $\text{rank}(\MM \MM^\top) = \text{rank}(\MM)$, we have $\text{rank}(\AA) = 2 \text{rank}(\MM)$. Thus, the rest of the eigenvalues of $\AA$ are 0 and there is a total of $d - 2d_1 = d_2 - d_1$ of them. Notice that \begin{align}
    \frac{d_1}{d_1 + d_2} = \frac{1}{\frac{d_1 + d_2}{d_1}} = \frac{1}{1+\frac{1}{r}}
\end{align}
Let $f_{+}(\lambda) = i\sqrt{\lambda}, f_{-}(\lambda) = -i\sqrt{\lambda}$, and let $(f_{+})_* \mu_{\MM \MM^\top}$ (resp., $(f_{-})_* \mu_{\MM \MM^\top}$) be the pushforward measure of $\mu_{\MM \MM^\top}$ by the function $f_{+}$ (resp., $f_{-}$).
Thus, by the definition of the pushforward measure (\autoref{def:pushforward}),
\begin{align} 
\begin{split} \label{eq:mu_a}
\mu_{\AA}(i\lambda) &= \left( 1 -\frac{2}{1+\frac{1}{r}} \right) \delta_0(\lambda) + \frac{1}{1+\frac{1}{r}} (f_{+})_* \mu_{\MM \MM^\top}(\lambda) + \frac{1}{1+\frac{1}{r}} (f_{-})_* \mu_{\MM \MM^\top}(\lambda)
\end{split}
\end{align}
We compute the pushforwards $(f_{+})_* \mu_{M M^\top}, (f_{-})_* \mu_{M M^\top}$ performing the change of variables $y = \pm i \sqrt{\lambda}$ under the assumption that $\mu_{M M^\top}(\lambda) = \rho_{M M^\top}(\lambda) d\lambda$:
\begin{align}
    \int_{\mathbb{R}_{\geq 0}} g \left(\pm i\sqrt{\lambda} \right) \dif \mu_{M M^\top}(\lambda) = \int_{\mathbb{R}_{\geq 0}} g \left(\pm i\sqrt{\lambda} \right) \rho_{M M^\top}(\lambda) d\lambda = \int_{\pm i \mathbb{R}_{\geq 0}} g \left(y \right) \rho_{M M^\top}(|y|^2) 2|y| \dif |y|,
\end{align}
which means that the density of $(f_{+})_* \mu_{M M^\top}$ at $y \in i\mathbb{R}_{\geq 0}$ is $2|y| \rho_{M M^\top}(|y|^2)$ and the density of $(f_{-})_* \mu_{M M^\top}$ at $y \in -i\mathbb{R}_{\geq 0}$ is also $2|y| \rho_{M M^\top}(|y|^2)$.
\end{proof}

\begin{restatable}{prop}{3termcond}
    \label{prop:3termcond}
    The condition 
    \begin{align} \label{eq:zeros_condition}
        \forall P, Q \text{ polynomials } \langle P(\lambda), \lambda Q(\lambda) \rangle = 0 \implies \langle \lambda P(\lambda), Q(\lambda) \rangle = 0
    \end{align} 
    is sufficient for any sequence $(P_k)_{k \geq 0}$ of orthogonal polynomials of increasing degrees to satisfy a three-term recurrence of the form
\begin{align} \label{eq:three_term_rec}
    \gamma_k P_k(\lambda) = (\lambda-\alpha_k) P_{k-1}(\lambda) - \beta_k P_{k-2}(\lambda), 
\end{align}
where
\begin{align} \label{eq:three_term_rec2}
    \gamma_k = \frac{\langle \lambda P_{k-1}(\lambda), P_k(\lambda) \rangle}{\langle P_k(\lambda), P_k(\lambda) \rangle}, \quad \alpha_k = \frac{\langle \lambda P_{k-1}(\lambda), P_{k-1}(\lambda) \rangle}{\langle P_{k-1}(\lambda), P_{k-1}(\lambda) \rangle}, \quad \beta_k = \frac{\langle \lambda P_{k-1}(\lambda), P_{k-2}(\lambda) \rangle}{\langle P_{k-2}(\lambda), P_{k-2}(\lambda) \rangle}
\end{align}
\end{restatable}
\begin{proof}
Since $\lambda P_{k-1}(\lambda)$ is a polynomial of degree $k$, and $(P_j)_{0 \leq j \leq k}$ is a basis of the polynomials of degree up to $k$, we can write
\begin{align} \label{eq:basis_decomposition}
    \lambda P_{k-1}(\lambda) = \sum_{j=0}^{k} \frac{\langle \lambda P_{k-1}, P_j \rangle}{\langle P_j, P_j \rangle} P_j(\lambda)
\end{align}
Now, remark that for all $j < k - 2$, $\langle P_{k-1}, \lambda P_j \rangle = 0$ because the inner product of $P_{k-1}$ with a polynomial of degree at most $k-2$. If we make use of the condition \eqref{eq:zeros_condition}, this implies that $\langle \lambda P_{k-1}, P_j \rangle = 0$ for all $j < k - 2$. Plugging this into \eqref{eq:basis_decomposition}, we obtain \eqref{eq:three_term_rec}.
\end{proof}

\begin{restatable}{prop}{evenmin} \label{prop:even_min}
Let $\Pi_{t}^{\mathbb{R}}$ be the set of polynomials with real coefficients and degree at most $t$. For $t \geq 0$ even, the minimum of the problem 
\begin{align} \label{eq:min_even_problem}
    \min_{P_{t} \in \Pi_{t}^{\mathbb{R}}, P_{t}(0)=1} \int_{i \mathbb{R} \setminus \{0\} } |P_t(\lambda)|^2 |\lambda| \rho_{\MM \MM^\top}(|\lambda|^2) \dif |\lambda|
\end{align}
is attained by an even polynomial with real coefficients.
\end{restatable}

\begin{proof}
Since $\dif \mu(i\lambda) \defas |\lambda| \rho_{M M^\top}(|\lambda|^2) \dif |\lambda|$ is supported in the imaginary axis and is symmetric with respect to $0$, for all polynomials $P,Q$,
\begin{align}
    \langle \lambda P(\lambda), Q(\lambda) \rangle = \int_{i\mathbb{R}} \lambda P(\lambda) Q(\lambda)^{*} d\mu(\lambda) = - \int_{i\mathbb{R}} P(\lambda) \lambda^* Q(\lambda)^{*} d\mu(\lambda) = - \langle P(\lambda), \lambda Q(\lambda) \rangle. 
\end{align}
Hence, $\langle P(\lambda), \lambda Q(\lambda) \rangle = 0$ implies $\langle \lambda P(\lambda), Q(\lambda) \rangle = 0$. By \autoref{prop:3termcond}, a three-term recurrence \eqref{eq:three_term_rec} and \eqref{eq:three_term_rec2} for the orthonormal sequence $(\phi_t)_{t \geq 0}$ of polynomials holds.

By \autoref{prop:orthonormal_even}, the orthonormal polynomials $(\phi_t)_{t \geq 0}$ of even (resp. odd) degree are even (resp. odd) and have real coefficients. Hence, for all $t \geq 0$ even
\begin{align}
    \frac{\sum_{k=0}^t \phi_k(\lambda) \phi_k(0)^*}{\sum_{k=0}^t |\phi_k(0)|^2} = \frac{\sum_{k=0}^{t/2} \phi_{2k}(\lambda) \phi_{2k}(0)^*}{\sum_{k=0}^{t/2} |\phi_{2k}(0)|^2}
\end{align}
is an even polynomial with real coefficients. By \autoref{thm:assche}, this polynomial attains the minimum of the problem 
\begin{align}
    \min_{P_{t} \in \Pi_{t}^{\mathbb{C}}, P_{t}(0)=1} \int_{i \mathbb{R} \setminus \{0\} } |P_t(\lambda)|^2 |\lambda| \rho_{M M^\top}(|\lambda|^2) \dif |\lambda|
\end{align}
and, a fortiori, the minimum of the problem in \eqref{eq:min_even_problem}, in which the minimization is restricted polynomials with real coefficients instead of complex coefficients.
\end{proof}

\begin{restatable}{prop}{orthonormaleven}
    \label{prop:orthonormal_even}
The polynomials $(\phi_t)_{t \geq 0}$ of the orthonormal sequence corresponding to the measure $\mu(i\lambda) = |\lambda| \rho_{M M^\top}(|\lambda|^2) d|\lambda|$ have real coefficients and are even (resp. odd) for even (resp. odd) $k$.
\end{restatable}
\begin{proof}
The proof is by induction. The base case follows from the choice $\phi_0 = 1$. Assuming that $\phi_{k-1} \in \mathbb{R}[X]$ by the induction hypothesis, we show that $\alpha_k = 0$ (where $\alpha_k$ is the coefficient from \eqref{eq:three_term_rec} and \eqref{eq:three_term_rec2}): 
\begin{align}
\begin{split}
    \langle \lambda \phi_{k-1}(\lambda), \phi_{k-1}(\lambda) \rangle 
    &= \int_{i \mathbb{R}} \lambda |\phi_{k-1}(\lambda)|^2 |\lambda| \rho_{M M^\top}(|\lambda|^2) d|\lambda| \\ &= \int_{ \mathbb{R}_{\geq 0}} i\lambda (|\phi_{k-1}(i\lambda)|^2 - |\phi_{k-1}(-i\lambda)|^2) \lambda \rho_{M M^\top}(\lambda^2) d\lambda = 0
\end{split}
\end{align}
The last equality follows from $|\phi_{k-1}(i\lambda)|^2 = |\phi_{k-1}(-i\lambda)|^2$, which holds because $\phi_{k-1}(i\lambda)^* = \phi_{k-1}(-i\lambda)$, and in turn this is true because $\phi_{k-1} \in \mathbb{R}[X]$ by the induction hypothesis.

Once we have seen that $\alpha_k = 0$, it is straightforward to apply the induction hypothesis once again to show that $\phi_k$ also satisfies the even/odd property. Namely, for $k$ even (resp. odd), $\gamma_k P_k = \lambda P_{k-1} - \beta_k P_{k-2}$, and the two polynomials in the right-hand side have even (resp. odd) degrees.

Finally, $\phi_k$ must have real coefficients because $\phi_{k-1}$ and $\phi_{k-2}$ have real coefficients by the induction hypothesis, and the recurrence coefficient $\beta_k$ is real, as
\begin{align}
\begin{split}
    \langle \lambda P_{k-1}(\lambda), P_{k-2}(\lambda) \rangle &= \int_{i \mathbb{R}} \lambda \phi_{k-1}(\lambda) \phi_{k-2}(\lambda)^*  |\lambda| \rho_{M M^\top}(|\lambda|^2) d|\lambda| \\ &= \int_{\mathbb{R}_{\geq 0}} i\lambda (\phi_{k-1}(i\lambda) \phi_{k-2}(i\lambda)^* - \phi_{k-1}(i\lambda)^* \phi_{k-2}(i\lambda)) \lambda \rho_{M M^\top}(\lambda^2) d\lambda \\ &= -\int_{\mathbb{R}_{\geq 0}} 2\lambda \text{Im}(\phi_{k-1}(i\lambda) \phi_{k-2}(i\lambda)^*)  \lambda \rho_{M M^\top}(\lambda^2) d\lambda \in \mathbb{R}.
\end{split}
\end{align}
\end{proof}

\begin{restatable}{prop}{twoproblems}
    \label{prop:twoproblems}
    Let $t \geq 0$ even. Assume that on $\mathbb{R}_{>0}$, the spectral density $\mu_{\MM \MM^\top}$ has Radon-Nikodym derivative $\rho_{\MM \MM^\top}$ with respect to the Lebesgue measure. If 
    \begin{align} \label{eq:q_star_t}
        Q^{\star}_{t/2} \defas \argmin_{\substack{P_{t/2} \in \Pi_{t/2}^{\mathbb{R}}, \\ P_{t/2}(0)=1}} \int_{\mathbb{R}_{>0}} P_{t/2}(\lambda)^2  \dif\mu_{-\AA^2}(\lambda),
    \end{align}
    and
    \begin{align} \label{eq:p_star_t}
        P^{\star}_t \defas \argmin_{\substack{P_{t} \in \Pi_{t}^{\mathbb{R}}, \\ P_{t}(0)=1}} \int_{i \mathbb{R} \setminus \{0\} } |P_t(\lambda)|^2 |\lambda| \rho_{\MM \MM^\top}(|\lambda|^2) \dif|\lambda|,
    \end{align}
    then $P^{\star}_t(\lambda) = Q^{\star}_{t/2}(-\lambda^2)$.
\end{restatable}

\begin{proof}
First, remark that the equalities in \eqref{eq:q_star_t} and \eqref{eq:p_star_t} are well defined because the $\argmin$ are unique by \autoref{thm:assche}.
Without loss of generality, assume that $d_x \leq d_y$ (otherwise switch the players), and let $r \defas d_x/d_y < 1$. Since,
\begin{align}
    -\AA^2 =
    \begin{bmatrix}
    \MM \MM^\top & 0 \\
    0 & \MM^\top \MM
    \end{bmatrix},
\end{align}
each eigenvalue of $\MM \MM^\top \in \mathbb{R}^{d_x \times d_x}$ is an eigenvalue of $-\AA^2$ with doubled duplicity, and the rest of eigenvalues are zero. Hence, we have $\mu_{-\AA^2} = \left(1 - 2/(1+\frac{1}{r})\right) \delta_0 + 2\mu_{\MM \MM^\top}/(1+\frac{1}{r}) $. Thus, for all $t \geq 0$,
\begin{align} \label{eq:q_star_eq}
  Q^{\star}_t = \argmin_{\substack{P_{t} \in \Pi_{t}^{\mathbb{R}}, \\ P_{t}(0)=1}} \int_{\mathbb{R}_{>0}} P_{t}(\lambda)^2  \dif\mu_{-\AA^2}(\lambda) = \argmin_{\substack{P_{t} \in \Pi_{t}^{\mathbb{R}}, \\ P_{t}(0)=1}} \int_{\mathbb{R}_{>0}} P_{t}(\lambda)^2  \rho_{\MM \MM^\top}(\lambda) \dif \lambda  
\end{align}
By \autoref{prop:even_min}, for an even $t \geq 0$ the minimum in \eqref{eq:p_star_t} is attained by an even polynomial with real coefficients. Hence,
\begin{align}
\begin{split} \label{eq:chain_min}
    &\min_{\substack{P_t \in \Pi_t^{\mathbb{R}}, \\ P_t(0)=1}} \int_{i \mathbb{R} \setminus \{0\} } |P_t(\lambda)|^2 |\lambda| \rho_{\MM \MM^\top}(|\lambda|^2) \dif|\lambda| = \min_{\substack{P_{t/2} \in \Pi_{t/2}^{\mathbb{R}}, \\ P_{t/2}(0)=1}} \int_{i \mathbb{R} \setminus \{0\} } |P_{t/2}(\lambda^2)|^2 |\lambda| \rho_{\MM \MM^\top}(|\lambda|^2) \dif|\lambda| \\ &= 2 \min_{\substack{P_{t/2} \in \Pi_{t/2}^{\mathbb{R}}, \\ P_{t/2}(0)=1}} \int_{\mathbb{R}_{>0}} |P_{t/2}((i\lambda)^2)|^2 \lambda \rho_{\MM \MM^\top}(\lambda^2) \dif \lambda = 2 \min_{\substack{P_{t/2} \in \Pi_{t/2}^{\mathbb{R}}, \\ P_{t/2}(0)=1}} \int_{\mathbb{R}_{>0}} P_{t/2}(\lambda^2)^2 \lambda \rho_{\MM \MM^\top}(\lambda^2) \dif \lambda \\ &= \min_{\substack{P_{t/2} \in \Pi_{t/2}^{\mathbb{R}}, \\ P_{t/2}(0)=1}} \int_{\mathbb{R}_{>0}} P_{t/2}(\lambda)^2 \rho_{\MM \MM^\top}(\lambda) \dif \lambda
\end{split}
\end{align}
Moreover, for any polynomial $Q_{t/2}$ that attains the minimum on the right-most term, the polynomial $P_{t}(\lambda) = Q_{t/2}(-\lambda^2)$ attains the minimum on the left-most term. In particular, using \eqref{eq:q_star_eq}, $P^{\star}_{t}(\lambda) \defas Q^{\star}_{t/2}(-\lambda^2)$ attains the minimum on the left-most term.
\end{proof}

\bilinearthmone*

\begin{proof}
Making use of \autoref{thm:expectation} and \autoref{prop:eigenvalue_dist}, we obtain that for any first-order method using the vector field $F$,
\begin{align} \label{eq:expectation_bilinear}
    \mathbb{E}[\dist(\xx_t,\mathcal{X}^{\star})] = R^2 \int_{\mathbb{C} \setminus \{0\}} |P_t(\lambda)|^2 \dif\mu_{\AA}(\lambda) = \frac{2 R^2}{1 + \frac{1}{r}} \int_{i \mathbb{R} \setminus \{0\} } |P_t(\lambda)|^2 |\lambda| \rho_{\MM \MM^\top}(|\lambda|^2) \dif|\lambda|
\end{align}

Let $Q^{\star}_{t/2}, P^{\star}_t$ be as defined in \eqref{eq:p_star_t} and \eqref{eq:q_star_t}. For $t \geq 0$ even the iteration $t$ of the average-case optimal method for the bilinear game must satisfy 
\begin{align} \label{eq:comparison_both_methods}
    \xx_t - P_{\mathcal{X}^\star}(\xx_0) = P^{\star}_t(\AA) (\xx_0 - P_{\mathcal{X}^\star}(\xx_0)) = Q^{\star}_{t/2}(-\AA^2) (\xx_0 - P_{\mathcal{X}^\star}(\xx_0))
\end{align}
On the other hand, the first-order methods for the minimization of the function $\frac{1}{2} \|F(\xx)\|^2$ make use of the vector field $\nabla \left( \frac{1}{2} \|F(\xx)\|^2\right) = \AA^\top(\AA \xx + \bb) = -\AA^2 (\xx - \xx^{\star})$. Let $\mu_{-\AA^2}$ be the spectral density of $-\AA^2$. By \autoref{thm:expectation}, the average-case optimal first-order method for the minimization problem is the one for which the residual polynomial $P_t$ (\autoref{prop:link_algo_polynomial}) minimizes the functional $\int_\RR P_t^2 \dif\mu_{-\AA^2}$. That is, the residual polynomial is $Q^{\star}_t$. From \eqref{eq:comparison_both_methods}, we see that the $t$-th iterate of the average-case optimal method for $F$ is equal to the $t/2$-th iterator of the average-case optimal method for $\nabla \left( \frac{1}{2} \|F(\xx)\|^2\right)$.
\end{proof}

\section{Proofs of \autoref{thm:opt_alg} and \autoref{thm:circular}}

\optalg*
\begin{proof}
We prove by induction that 
\begin{align} \label{eq:induction_statement}
    \xx_t - \xx^{\star} = \frac{\sum_{k=0}^t \phi_k(\AA) \phi_k(0)^*}{\sum_{k=0}^t \phi_k(0)^2} (\xx_0-\xx^{\star})
\end{align}
The base step $t = 0$ holds trivially because $\phi_0 = 1$. Assume that \eqref{eq:induction_statement} holds for $t-1$. Subtracting $\xx^{\star}$ from \eqref{eq:opti_algorithm_disk}, we have
\begin{align}
\begin{split} \label{eq:intermediate_eq}
    \xx_t - \xx^{\star} = \frac{\sum_{k=0}^{t-1} \phi_k(0)^2 }{\sum_{k=0}^{t} \phi_k(0)^2} (\xx_{t-1} - \xx^{\star}) + \frac{\phi_t(0)^2}{\sum_{k=0}^{t} \phi_k(0)^2} (\yy_t - \xx^{\star})
\end{split}
\end{align}
If
\begin{align} \label{eq:y_equality}
    \phi_t(0)^2 (\yy_t - \xx^{\star}) = \phi_t(0) \phi_t(\AA)(\xx_0 - \xx^{\star}),
\end{align}
by the induction hypothesis for $t-1$ and \eqref{eq:intermediate_eq}, we have
\begin{align}
\begin{split}
    \xx_t - \xx^{\star} &= \frac{\sum_{k=0}^{t-1} \phi_t(0) \phi_t(\AA)}{\sum_{k=0}^t \phi_k(0)^2} (\xx_0-\xx^{\star}) + \frac{\phi_t(0) \phi_t(\AA)}{\sum_{k=0}^{t} \phi_k(0)^2} (x_0 - x_*) \\ &= \frac{\sum_{k=0}^t \phi_t(0) \phi_t(\AA)}{\sum_{k=0}^t \phi_k(0)^2} (x_0-x_*),
\end{split}
\end{align}
which concludes the proof of \eqref{eq:induction_statement}. The only thing left is to show \eqref{eq:y_equality}, again by induction. The base case follows readily from $\yy_0 = \xx_0$ in \eqref{eq:opti_algorithm_disk}. Dividing by $\phi_t(0)^2$, we rewrite \eqref{eq:y_equality} as
\begin{align}
    \yy_t - \xx^{\star} = \frac{\phi_t(\AA)}{\phi_t(0)}(\xx_0 - \xx^{\star}) = \psi_t(\AA) (\xx_0 - \xx^{\star}),
\end{align}
where $\psi_t$ is the $t$-th orthogonal residual polynomial of sequence. By \autoref{ass:three_term_residual}, $\psi_t$ must satisfy the recurrence in \eqref{eq:three_term_residual_eq}.
If we subtract $x_*$ from the second line of \eqref{eq:opti_algorithm_disk}, we apply the induction hypothesis and then the recurrence in \eqref{eq:three_term_residual_eq}, we obtain
\begin{align}
\begin{split} \label{eq:y_iterates}
    \yy_t - \xx^{\star} &= a_t (\yy_{t-1}-\xx^{\star}) + (1- a_t) (\yy_{t-2}-\xx^{\star}) + b_t F(\yy_{t-1}) \\ &= a_t (\yy_{t-1}-\xx^{\star}) + (1- a_t) (\yy_{t-2}-\xx^{\star}) + b_t \AA(\yy_{t-1}-\xx_*) \\ &= a_t \psi_{t-1}(\AA) (\xx_0 - \xx^{\star}) + (1- a_t) \psi_{t-2}(\AA) (\xx_0 - \xx^{\star}) + b_t \AA\psi_{t-1}(\AA) (\xx_0 - \xx^{\star}) \\ &= \psi_t(\AA) (\xx_0 - \xx^{\star}),
\end{split}
\end{align}
thus concluding the proof of \eqref{eq:y_equality}.
\end{proof}

\begin{restatable}{prop}{orthonormalcentered}
    \label{prop:orthonormal_centered}
Suppose that \autoref{ass:circular_measure} holds with $C=0$, that is, the circular support of $\mu$ is centered at $0$. Then, the basis of orthonormal polynomials for the scalar product
\begin{align}
    \langle P, Q \rangle = \int_{D_{R,0}} P(\lambda) Q(\lambda)^* \dif\mu(\lambda) \quad \text{is} \quad
    \phi_k(\lambda) = \frac{\lambda^k}{D_{k,R}}, \quad \forall k \geq 0,
\end{align}
where $K_{k,R} = \sqrt{2\pi \int_{0}^{R} r^{2k} d\mu_{R}(r)}$.
\end{restatable}
\begin{proof}
First, we will show that if $\mu$ satisfies \autoref{ass:circular_measure} with $C=0$, then $\langle \lambda^i, \lambda^j \rangle = 0$ if $j, k \geq 0$ with $j \neq k$ (without loss of generality, suppose that $j > k$).
\begin{align}
\begin{split}
    \langle \lambda^j, \lambda^k \rangle &= \int_{D_{R,0}} \lambda^j (\lambda^*)^k \dif \mu(\lambda) = \int_{D_{R,0}} \lambda^{j-k} |\lambda|^{2k} \dif \mu(\lambda) \\ &= \int_0^R \frac{1}{2\pi} \int_0^{2\pi} (r e^{i\theta})^{j-k} r^{2k} \dif \theta \dif \mu_R(r) = \frac{1}{2\pi} \int_0^{2\pi} e^{i\theta(j-k)} \dif \theta \int_0^R r^{j+k} \dif \mu_R(r) \\ &= \frac{e^{i2\pi}-1}{2\pi i(j-k)}\int_0^R r^{j+k} \dif \mu_R(r) = 0
\end{split}
\end{align}
And for all $k \geq 0$,
\begin{align}
    \langle \lambda^k, \lambda^k \rangle = \int_{D_{R,0}} |\lambda^k|^2 \dif \mu(\lambda) = \int_0^R \frac{1}{2\pi} \int_0^{2\pi} r^{2k} \dif \theta \dif \mu_R(r) = \int_0^{2\pi} r^{2k} \dif \mu_R(r).
\end{align}
\end{proof}

\orthocircular*
\begin{proof}
The result follows from \autoref{prop:orthonormal_centered} using the change of variables $z \rightarrow z + C$. To compute the measure $\mu_R$ for the uniform measure on $D_{C,R}$, we perform a change of variables to circular coordinates:
\begin{align}
\begin{split}
    \int_{D_{C,R}} f(\lambda) \dif \mu(\lambda) &= \frac{1}{\pi R^2} \int_0^R \int_{0}^{2\pi} f(C + r e^{i\theta}) r \dif\theta \dif r = \int_0^R \int_{0}^{2\pi} f(C + r e^{i\theta}) \dif\theta \dif \mu_R(r). \\
    &\implies \dif \mu_R(r) = \frac{r}{\pi R^2} \dif r
\end{split}
\end{align}
And 
\begin{align}
    \int_0^R r^{2t} \dif \mu_R(r) = \frac{1}{\pi R^2} \int_0^R r^{2t+1} \dif r = \frac{1}{\pi} \frac{R^{2t}}{2t+2} \implies K_{t,R} = R^t/\sqrt{t+1}.
\end{align}
\end{proof}

\circularthmone*
\begin{proof}
By \autoref{prop:ortho_circular}, the sequence of residual orthogonal polynomials is given by $\psi_t(\lambda) = \phi_t(\lambda)/\phi_t(0) = \left(1 - \frac{\lambda}{C}\right)^t$. Hence, \autoref{ass:three_term_residual} is fulfilled with $a_t = 1, b_t = - \frac{1}{C}$, as $\psi_t(\lambda) = \psi_{t-1}(\lambda) - \frac{\lambda}{C} \psi_{t-1}(\lambda)$.
We apply \autoref{thm:opt_alg} and make use of the fact that $\phi_k(0)^2 = \frac{C^{2k}}{K_{t,R}^2}$. See \autoref{prop:convrates} for the rate on $\dist(\xx_t,\mathcal{X}^\star)$.
\end{proof}

\section{Proof of \autoref{thm:asymptotic2}}
\begin{restatable}{prop}{limitprop}
    \label{prop:limitprop}
Suppose that the assumptions of \autoref{thm:circular} hold with the probability measure $\mu_R$ fulfilling $\mu_R([r,R]) = \Omega((R-r)^\kappa)$ for $r$ in $[r_0,R]$ for some $r_0 \in [0,R)$ and for some $\kappa \in \mathbb{Z}$. Then, 
\begin{align}
\lim_{t \rightarrow \infty} \frac{\frac{C^{2t}}{K_{t,R}^{2}}}{\sum_{k=0}^{t} \frac{C^{2k}}{K_{k,R}^{2}}} = 1 - \frac{R^2}{C^2}\,.
\end{align}
\end{restatable}
\begin{proof}

Given $\epsilon > 0$, let $c_\epsilon \in \mathbb{Z}_{\geq 0}$ be the minimum such that 
\begin{align} \label{eq:cepsilon_def}
     \frac{1}{\sum_{i=0}^{c_\epsilon} \left( \frac{R^2}{C^2} \right)^i} \leq (1+\epsilon) \frac{1}{\sum_{i=0}^{\infty} \left( \frac{R^2}{C^2} \right)^i} = (1+\epsilon) \left(1 - \frac{R^2}{C^2} \right) 
\end{align}
Define $Q_{t,R} \defas \frac{R^{2t}}{K_{t,R}^2}$. Then, 
\begin{align} \label{eq:final_equality}
    \frac{\frac{C^{2t}}{K_{t,R}^{2}}}{\sum_{k=0}^{t} \frac{C^{2k}}{K_{k,R}^{2}}} = \frac{\frac{C^{2t}}{R^{2t}} Q_{t,R} }{\sum_{k=0}^{t} \frac{C^{2k}}{R^{2k}} Q_{k,R}} = \frac{ Q_{t,R} }{\sum_{k=0}^{t} \left(\frac{R^2}{C^2}\right)^{t-k} Q_{k,R}}
\end{align}
Now, on one hand, using that $Q_{t,R}$ is an increasing sequence on $t$,
\begin{align} \label{eq:final_lower_bound}
    \frac{ Q_{t,R} }{\sum_{k=0}^{t} \left(\frac{R^2}{C^2}\right)^{t-k} Q_{k,R}} \geq \frac{ 1}{\sum_{k=0}^{t} \left( \frac{R^2}{C^2}\right)^{t-k}} \geq \frac{ 1}{\sum_{k=0}^{\infty} \left( \frac{R^2}{C^2}\right)^{k}} = 1 - \frac{R^2}{C^2}
\end{align}
On the other hand, for $t \geq c_{\epsilon}$,
\begin{align} \label{eq:upper_b}
    \frac{ Q_{t,R} }{\sum_{k=0}^{t} \left(\frac{R^2}{C^2}\right)^{t-k} Q_{k,R}} \leq \frac{ Q_{t,R} }{\sum_{k=t-c_{\epsilon}}^{t} \left(\frac{R^2}{C^2}\right)^{t-k} Q_{k,R}} = \frac{ Q_{t,R} }{\sum_{k=t-c_{\epsilon}}^{t} \left(\frac{R^2}{C^2}\right)^{t-k} \left(Q_{t,R} - \int_{k}^t \frac{d}{ds}Q_{s,R} \dif s \right)}
\end{align}
Thus, we want to upper-bound $\int_{k}^t \frac{d}{ds}Q_{s,R} \dif s$.
First, notice that
\begin{align}
    \frac{d}{ds}Q_{s,R} = \frac{d}{ds} \left( \int_0^R \left( \frac{r}{R} \right)^{2s} \dif \mu_R(r) \right)^{-1} = \frac{\int_0^R \left( \frac{r}{R} \right)^{2s} \left( - \log(\frac{r}{R}) \right) \dif \mu_R(r)}{\left( \int_0^R \left( \frac{r}{R} \right)^{2s} \dif \mu_R(r) \right)^{2}}
\end{align}
By concavity of the logarithm function we obtain $\log (\frac{R}{r}) \leq \frac{R}{r_0}-1$ for $r \in [r_0,R]$. Choose $r_0$ close enough to $R$ so that $\frac{R}{r_0} - 1 \leq \epsilon/c_\epsilon$. We obtain that
\begin{align} 
    \int_0^R \left( \frac{r}{R} \right)^{2s} \log \left(\frac{R}{r} \right) \dif \mu_R(r) \leq \int_0^{r_0} \left( \frac{r}{R} \right)^{2s} \log \left(\frac{R}{r} \right) \dif \mu_R(r) + \int_{r_0}^R \left( \frac{r}{R} \right)^{2s} \left(\frac{R}{r_0} - 1\right) \dif \mu_R(r).
\end{align}
Thus,
\begin{align}
\begin{split} \label{eq:int_bound}
    \int_{k}^t \frac{d}{ds}Q_{s,R} \dif s &\leq \int_{k}^t \frac{ \int_0^{r_0} \left( \frac{r}{R} \right)^{2s} \log \left(\frac{R}{r} \right) \dif \mu_R(r) }{\left( \int_0^R \left( \frac{r}{R} \right)^{2s} \dif \mu_R(r) \right)^{2}} \dif s + \int_{k}^t \frac{ \int_{r_0}^R \left( \frac{r}{R} \right)^{2s} \left(\frac{R}{r_0} - 1\right) \dif \mu_R(r)}{\left( \int_0^R \left( \frac{r}{R} \right)^{2s} \dif \mu_R(r) \right)^{2}} \dif s.
\end{split}
\end{align}
Using that $\log x \leq x$, for $k \in [t - c_{\epsilon}, t]$ we can bound the first term of \eqref{eq:int_bound} as
\begin{align} 
\begin{split} \label{eq:first_term_bound}
    \int_{k}^t \frac{ \int_0^{r_0} \left( \frac{r}{R} \right)^{2s} \log \left(\frac{R}{r} \right) \dif \mu_R(r) }{\left( \int_0^R \left( \frac{r}{R} \right)^{2s} \dif \mu_R(r) \right)^{2}} \dif s &\leq \int_{k}^t \frac{ \int_0^{r_0} \left( \frac{r}{R} \right)^{2s-1}  \dif \mu_R(r) }{\left( \int_0^R \left( \frac{r}{R} \right)^{2s} \dif \mu_R(r) \right)^{2}} \dif s \\ &\leq (t-k) \frac{\left( \frac{r_0}{R} \right)^{2k-1}}{\left( \int_0^R \left( \frac{r}{R} \right)^{2t} \dif \mu_R(r) \right)^{2}} \\ &\leq c_{\epsilon} \left( \frac{r_0}{R} \right)^{2(t-c_{\epsilon})-1} Q_{t,R}^2 \\ &\leq c_{\epsilon} \left( \frac{r_0}{R} \right)^{2(t-c_{\epsilon})-1} \frac{1}{(c_1)^2} (2t+1)^{2 \kappa} \xrightarrow{t \rightarrow \infty} 0.
\end{split}
\end{align}
In the last inequality we use that by \autoref{prop:limitprop2}, for $t$ large enough, $Q_{t,R} = \frac{R^{2t}}{K_{t,R}^2} \leq (2t+1)^k / c_1$. For $k \in [t - c_{\epsilon}, t]$, the second term of \eqref{eq:int_bound} can be bounded as
\begin{align}
\begin{split} \label{eq:second_term_bound}
    \int_{k}^t \frac{ \int_{r_0}^R \left( \frac{r}{R} \right)^{2s} \frac{R}{r_0} \dif \mu_R(r)}{\left( \int_0^R \left( \frac{r}{R} \right)^{2s} \dif \mu_R(r) \right)^{2}} \dif s &\leq (t - k) \left(\frac{R}{r_0} - 1\right) \frac{1}{\int_0^R \left( \frac{r}{R} \right)^{2t} \dif \mu_R(r)} \\ &\leq c_{\epsilon} \left(\frac{R}{r_0} - 1\right) \frac{1}{\int_0^R \left( \frac{r}{R} \right)^{2t} \dif \mu_R(r)} \\& \leq \epsilon Q_{t,R}.
\end{split}
\end{align}
From \eqref{eq:int_bound}, \eqref{eq:first_term_bound} and \eqref{eq:second_term_bound}, we obtain that for $t$ large enough, for $k \in [t - c_{\epsilon}, t]$,
\begin{align} \label{eq:Q_derivative}
    \int_{k}^t \frac{d}{ds}Q_{s,R} \dif s \leq 2 \epsilon Q_{t,R}.
\end{align}
Hence, we can bound the right-hand side of \eqref{eq:upper_b}:
\begin{align}
\begin{split} \label{eq:final_upper_b}
    &\frac{ Q_{t,R} }{\sum_{k=t-c_{\epsilon}}^{t} \left(\frac{R^2}{C^2}\right)^{t-k} \left(Q_{t,R} - \int_{k}^t \frac{d}{ds}Q_{s,R} \dif s \right)} \leq \frac{ Q_{t,R} }{\sum_{k=t-c_{\epsilon}}^{t} \left(\frac{R^2}{C^2}\right)^{t-k} \left(Q_{t,R} - 2 \epsilon Q_{t,R} \right)} \\ &= \frac{ 1 }{(1-2\epsilon) \sum_{k=t-c_{\epsilon}}^{t} \left(\frac{R^2}{C^2}\right)^{t-k}} = \frac{ 1 }{(1-2\epsilon) \sum_{k=0}^{c_\epsilon} \left(\frac{R^2}{C^2}\right)^{k}} \leq \frac{1+\epsilon}{1-2\epsilon} \left( 1 - \frac{R^2}{C^2} \right).
\end{split}
\end{align}
The last inequality follows from the definition of $c_\epsilon$ in \eqref{eq:cepsilon_def}. Since $\epsilon$ is arbitrary, by the sandwich theorem applied on \eqref{eq:final_equality}, \eqref{eq:final_lower_bound} and \eqref{eq:final_upper_b},
\begin{align}
\lim_{t \rightarrow \infty} \frac{\frac{C^{2t}}{K_{t,R}^{2}}}{\sum_{k=0}^{t} \frac{C^{2k}}{K_{k,R}^{2}}} = 1 - \frac{R^2}{C^2}.
\end{align}
\end{proof}

\begin{restatable}{prop}{limitprop2}
    \label{prop:limitprop2}
Under the assumptions of \autoref{thm:circular}, we have that there exists $c_1 > 0$ such that for $t$ large enough,
\begin{align}
K_{t,R}^2 \geq c_1 R^{2t} (2t + 1)^{-\kappa}.
\end{align}
\end{restatable}
\begin{proof}
By the assumption on $\mu_R$, there exist $r_0, c_1, \kappa > 0$ such that
\begin{align}
\begin{split} \label{eq:ktreq_lower}
    K_{t,R}^2 &\defas 2 \pi \int_{0}^R r^{2t} \dif \mu_{R}(r) = 2 \pi \int_{0}^{r_0} r^{2t} \dif \mu_{R}(r) + 2 \pi \int_{r_0}^{R} r^{2t} \dif \mu_{R}(r) \\ &\geq 2 \pi c_1 \int_{r_0}^{R} r^{2t} (R-r)^{\kappa-1} \dif r = - 2 \pi c_1 \int_0^{r_0} r^{2t} (R-r)^{\kappa-1} \dif r + 2 \pi c_1 \int_{0}^{R} r^{2t} (R-r)^{\kappa-1} \dif r \\ &\geq -2 \pi c_1 R r_0^{2t} + 2 \pi c_1 R^{2t + \kappa} B(2t + 1,\kappa).
\end{split}
\end{align}
where the beta function $B(x,y)$ is defined as
\begin{align}
    B(x,y) \defas \int_0^1 r^{x+1} (1-r)^{y+1} \dif r. 
\end{align}
Using the link between the beta function and the gamma function $B(x,y) = \Gamma(x) \Gamma(y) / \Gamma(x+y)$, and Stirling's approximation, we obtain that for fixed $y$ and large $x$,
\begin{align}
    B(x,y) \sim \Gamma(y) x^{-y}.
\end{align}
Hence, for $t$ large enough, $B(2t + 1,\kappa) \sim  \Gamma(\kappa) (2t + 1)^{-\kappa} = (\kappa-1)! (2t + 1)^{-\kappa}$. Hence, from \eqref{eq:ktreq_lower} 
we obtain that there exist $c'_1$ 
depending only on $\kappa$ and $r_0$ such that for $t$ large enough
\begin{align} 
\begin{split} \label{eq:lower_bound_K}
    K_{t,R}^2 &\geq -2 \pi c_1 R r_0^{2t} + 2 \pi c_1 R^{2t + \kappa} (k-1)! (2t + 1)^{-\kappa} \geq c'_1 R^{2t} (2t + 1)^{-\kappa}.
\end{split}
\end{align}
\end{proof}

\asymptoticcirc*
\begin{proof}
The proof follows directly from \autoref{thm:circular} and \autoref{prop:limitprop}. See \eqref{eq:rate_asymp} and \eqref{eq:limiting_rates} in \autoref{prop:convrates} for the statement regarding the convergence rate.
\end{proof}

\begin{restatable}{prop}{convrates}
    \label{prop:convrates}
For the average-case optimal algorithm \eqref{eq:opti_algorithm_disk2},
\begin{align} \label{eq:rate_opt}
\mathbb{E} \dist(\xx_t,\mathcal{X}^{\star}) = \xi_{opt}(t) \defas \frac{1}{\sum_{k=0}^t \frac{C^{2k}}{K_{k,R}^2}} 
\end{align}
For the average-case asymptotically optimal algorithm \eqref{eq:opt_circ_asymp},
\begin{align} \label{eq:rate_asymp}
\mathbb{E} \dist(\xx_t,\mathcal{X}^{\star}) = \xi_{asymp}(t) \defas \left( 1- \left(\frac{R}{C} \right)^2 \right)^2 \sum_{k=1}^t \frac{K_{k,R}^2}{C^{2k}} \left( \frac{R}{C} \right)^{4(t - k)} + \left( \frac{R}{C} \right)^{4t}
\end{align}
For the iterates $\yy_t$ in \eqref{eq:opti_algorithm_disk2}, i.e. gradient descent with stepsize $1/C$, we have
\begin{align} \label{eq:rate_gd}
\mathbb{E} \dist(\yy_t,\mathcal{X}^{\star}) = \xi_{GD}(t) \defas \frac{K_{t,R}^2}{C^{2t}}
\end{align}
Moreover, for all $t \geq 0$, we have $\xi_{opt}(t) \leq \xi_{asymp}(t)$, and under the assumptions of \eqref{thm:asymptotic1}, 
\begin{align} \label{eq:limiting_rates}
    \lim_{t \rightarrow \infty} \frac{\xi_{opt}(t)}{\xi_{asymp}(t)} = 1, \quad \lim_{t \rightarrow \infty} \frac{\xi_{opt}(t)}{\xi_{GD}(t)} = \frac{\xi_{asymp}(t)}{\xi_{GD}(t)} = 1- \left( \frac{R}{C} \right)^2
\end{align}
\end{restatable}

\begin{proof}
To show \eqref{eq:rate_opt}, \eqref{eq:rate_asymp}, \eqref{eq:rate_gd}, we use the expression $\xx_t - \xx^{\star} = P_t(\AA) (\xx_0-\xx^{\star})$  (\autoref{prop:link_algo_polynomial}) and then evaluate $\|P_t\|_{\mu}^2 =\int_{\mathbb{C} \setminus \{0\}} |P_t|^2 \ \dif \mu$ (\autoref{thm:expectation}).

For \eqref{eq:rate_opt}, the value of $\|P_t\|_{\mu}^2$ follows directly from \autoref{thm:assche}, which states that the value for the optimal residual polynomial $P_t$ is
\begin{align}
    \frac{1}{\sum_{k=0}^t |\phi_k(0)|^2} = \frac{1}{\sum_{k=0}^t \frac{C^{2k}}{K_{k,R}^2}}.
\end{align}

A simple proof by induction shows that for the asymptotically optimal algorithm \eqref{eq:opt_circ_asymp}, the following expression holds for all $t \geq 0$:
\begin{align}
    \xx_t - \xx^{\star} = \left(\left( \frac{R}{C} \right)^{2t} + \left(1 - \left( \frac{R}{C} \right)^2 \right) \sum_{k=1}^t \left(1-\frac{\AA}{C} \right)^k \left( \frac{R}{C} \right)^{2(t-k)} \right) (\xx_0-\xx^{\star})
\end{align}
Thus, 
\begin{align}
\begin{split}
    P_t(\lambda) &= \left( \frac{R}{C} \right)^{2t} + \left(1 - \left( \frac{R}{C} \right)^2 \right) \sum_{k=1}^t \left(1-\frac{\lambda}{C} \right)^k \left( \frac{R}{C} \right)^{2(t-k)} \\ &= \left( \frac{R}{C} \right)^{2t} \phi_0(\lambda) + \left(1 - \left( \frac{R}{C} \right)^2 \right) \sum_{k=1}^t \frac{K_{k,R}}{C^k} \phi_k(\lambda) \left( \frac{R}{C} \right)^{2(t-k)},
\end{split}
\end{align}
which concludes the proof of \eqref{eq:rate_asymp}, as
\begin{align}
    \|P_t\|_{\mu}^2 = \left( 1- \left(\frac{R}{C} \right)^2 \right)^2 \sum_{k=1}^t \frac{K_{k,R}^2}{C^{2k}} \left( \frac{R}{C} \right)^{4(t - k)} + \left( \frac{R}{C} \right)^{4t}.
\end{align}
By equation \eqref{eq:y_iterates},
\begin{align}
    \yy_t - \xx^{\star} = \left(1 - \frac{\AA}{C} \right)^t (\yy_0 - \xx^{\star}) = \frac{K_{t,R}}{C^t} \phi_k(\AA) (\yy_0 - \xx^{\star})
\end{align}
Thus, for the $\yy_t$ iterates, $\|P_t\|_\mu^2 = \frac{K_{t,R}^2}{C^{2t}}$, and \eqref{eq:rate_gd} follows.

Now, $\xi_{opt}(t) \leq \xi_{asymp}(t), \forall t \geq 0$ is a consequence of $\xi_{opt}(t)$ being the rate of the optimal algorithm. And
\begin{align}
    \lim_{t \rightarrow \infty} \frac{\xi_{opt}(t)}{\xi_{GD}(t)} = \lim_{t \rightarrow \infty} \frac{\frac{C^{2t}}{K_{t,R}^{2}}}{\sum_{k=0}^{t} \frac{C^{2k}}{K_{k,R}^{2}}} = 1 - \frac{R^2}{C^2}
\end{align}

follows from \autoref{prop:limitprop}. 
To show $\lim_{t \rightarrow \infty} \frac{\xi_{opt}(t)}{\xi_{GD}(t)} = 1 - \frac{R^2}{C^2}$, which concludes the proof, we rewrite
\begin{align} \label{eq:asymp_rewrite}
\xi_{asymp}(t) = \left( \frac{R}{C} \right)^{2t} \left( \left( 1- \left(\frac{R}{C} \right)^2 \right)^2 \sum_{k=1}^t \frac{1}{Q_{k,R}} \left( \frac{R}{C} \right)^{2(t - k)} + \left( \frac{R}{C} \right)^{2t} \right),
\end{align}
using that by definition, $Q_{k,R} = R^{2k}/K_{k,R}^2$. Now, let $c_{\epsilon} \in \mathbb{Z}_{\geq 0}$ such that
\begin{align}
    \sum_{k=c_{\epsilon}}^{\infty} \left(\frac{R}{C} \right)^{2k} \leq \epsilon.
\end{align}
Using the same argument as in \autoref{prop:limitprop} (see \eqref{eq:Q_derivative}), for $t$ large enough and $k \in [t- c_{\epsilon}, t]$,
\begin{align} \label{eq:Q_derivative2}
    \int_{k}^t \frac{d}{ds}Q_{s,R} \dif s \leq 2 \epsilon Q_{t,R}.
\end{align}
Hence, for $t$ large enough,
\begin{align}
\begin{split}
    &\left( 1- \left(\frac{R}{C} \right)^2 \right)^2 \sum_{k=1}^t \frac{1}{Q_{k,R}} \left( \frac{R}{C} \right)^{2(t - k)} + \left( \frac{R}{C} \right)^{2t}  \\ &=  \left( 1- \left(\frac{R}{C} \right)^2 \right)^2 \left( \sum_{k=t-c_{\epsilon}}^t \frac{1}{Q_{t,R} - \int_{k}^t \frac{d}{ds}Q_{s,R}} \left( \frac{R}{C} \right)^{2(t - k)} + \sum_{k=1}^{t-c_{\epsilon}} \frac{1}{Q_{k,R}} \left( \frac{R}{C} \right)^{2(t - k)} \right) + \left( \frac{R}{C} \right)^{2t} \\ &\leq \left( 1- \left(\frac{R}{C} \right)^2 \right)^2 \left( \frac{1}{(1-2\epsilon)Q_{t,R}} \sum_{k=t-c_{\epsilon}}^t \left( \frac{R}{C} \right)^{2(t - k)} + \sum_{k=1}^{t-c_{\epsilon}} \left( \frac{R}{C} \right)^{2(t - k)} \right) + \epsilon \\ &\leq \left( 1- \left(\frac{R}{C} \right)^2 \right) \left( \frac{1}{(1-2\epsilon)Q_{t,R}} + \left( 1- \left(\frac{R}{C} \right)^2 \right) \epsilon \right) + \epsilon,
\end{split}
\end{align}
which can be made arbitrarily close to $\left( 1- \left(\frac{R}{C} \right)^2 \right) \frac{1}{Q_{t,R}}$ by taking $\epsilon > 0$ small enough. Plugging this into \eqref{eq:asymp_rewrite}, we obtain that we can make $\xi_{asymp}(t)$ arbitrarily close to $\left( 1- \left(\frac{R}{C} \right)^2 \right) \left(\frac{R}{C} \right)^{2t} \frac{1}{Q_{t,R}} = \left( 1- \left(\frac{R}{C} \right)^2 \right) \xi_{GD}(t)$ by taking $t$ large enough.
\end{proof}

\end{document}

%% file: math_commands.tex
\usepackage{amsmath,amsfonts,bm}

\def\1{\bm{1}}

\def\vv{{\bm{v}}}

\DeclareMathAlphabet{\mathsfit}{\encodingdefault}{\sfdefault}{m}{sl}
\SetMathAlphabet{\mathsfit}{bold}{\encodingdefault}{\sfdefault}{bx}{n}

\newcommand{\E}{\mathbb{E}}

\newcommand{\R}{\mathbb{R}}



%% file: preamble_2.tex
\usepackage{amsmath}
\usepackage{amssymb}
\usepackage{empheq}
\usepackage{xcolor, color, colortbl}
\usepackage{framed}
\usepackage{pifont}
\usepackage{enumitem}
\usepackage{graphicx} 

\usepackage{nicefrac}
\usepackage{booktabs}
\usepackage{multirow}
\usepackage{rotating}
\usepackage{nccmath}

\usepackage[tikz]{bclogo}
\usepackage{wasysym}

\definecolor{mydarkblue}{rgb}{0,0.08,0.45}
\usepackage[
    colorlinks=true,
    linkcolor=mydarkblue,
    citecolor=mydarkblue,
    filecolor=mydarkblue,
    urlcolor=mydarkblue,
    allcolors=mydarkblue,
    pdfview=FitH]{hyperref}

\usepackage{url}
\usepackage{relsize}

\def\xx{{\boldsymbol x}}
\def\yy{{\boldsymbol y}}
\def\zz{{\boldsymbol z}}

\def\bb{{\boldsymbol b}}

\def\II{{\boldsymbol I}}
\def\yy{{\boldsymbol y}}
\def\vv{{\boldsymbol v}}
\def\uu{{\boldsymbol u}}
\def\ww{{\boldsymbol w}}
\def\zz{{\boldsymbol z}}

\def\BB{{\boldsymbol B}}
\def\AA{{\boldsymbol A}}
\def\CC{{\boldsymbol C}}
\def\MM{{\boldsymbol M}}
\def\DD{{\boldsymbol D}}

\def\QQ{{\boldsymbol Q}}
\def\UU{{\boldsymbol U}}
\def\HH{{\boldsymbol H}}
\def\thetaa{{\boldsymbol \theta}}

\def\dif{\mathop{}\!\mathrm{d}}

\def\RR{{\mathbb R}}
\def\EE{{\mathbb{E}\,}}

\renewcommand{\gg}{\boldsymbol{g}}

\DeclareMathOperator{\tr}{tr}
\def\defas{\stackrel{\text{def}}{=}}
\DeclareMathOperator*{\dist}{dist}

\DeclareMathOperator*{\Span}{\mathbf{span}}

\newcommand*\mybluebox[1]{\colorbox{myblue}{\hspace{1em}#1\hspace{1em}}}

\DeclareMathOperator*{\argmin}{{arg\,min}}

\definecolor{myblue}{HTML}{D2E4FC}
\definecolor{Gray}{gray}{0.92}
    
\newtheorem{assumption}{Assumption}
\newtheorem{definition}{Definition}

\usepackage{thmtools}
\usepackage{thm-restate}

\usepackage{wrapfig}

\usepackage{natbib}
\usepackage{mathtools}
\mathtoolsset{showonlyrefs}